\documentclass{amsart}

\usepackage{graphicx}
\usepackage{amssymb, comment}
\usepackage{caption}

\title[Multiplicity of periodic orbits and homoclinics near critical levels]{On the multiplicity of periodic orbits and homoclinics near critical energy levels of Hamiltonian systems in $\mathbb{R}^4$}

\author{Naiara V. de Paulo}
\address[Naiara V. de Paulo]{Universidade Federal de Santa Catarina,  Departamento de Mate\-m\'atica -- Rua Pomerode, 710 - Salto do Norte - Blumenau SC, Brazil 89065-300.}
\email{naiara.vergian@ufsc.br}

\author{Pedro A. S. Salom\~ao}
\address[Pedro A. S. Salom\~ao]{Universidade de S\~ao Paulo,  Instituto de Matem\'atica e Estat\'istica -- Departamento de Matem\'atica, Rua do Mat\~ao, 1010 - Cidade Universit\'aria - S\~ao Paulo SP, Brazil 05508-090.}
\email{psalomao@ime.usp.br}

\newcommand{\R}{\mathbb{R}}
\newcommand{\Z}{\mathbb{Z}}

\newcommand{\N}{\mathbb{N}}

\newcommand{\A}{\mathcal{A}}

\newcommand{\F}{\mathcal{F}}

\newcommand{\U}{\mathcal U}

\newcounter{newcounter}[section]

\numberwithin{equation}{section}
\numberwithin{newcounter}{section}
\numberwithin{figure}{section}
\numberwithin{footnote}{section}

\newtheorem{definition}[newcounter]{Definition}

\newtheorem{lem}[newcounter]{Lemma}
\newtheorem{prop}[newcounter]{Proposition}

\newtheorem{theo}[newcounter]{Theorem}

\begin{document}

\begin{abstract} We study two-degree-of-freedom Hamiltonian systems. Let us assume that the zero energy level of a real-analytic Hamiltonian function $H:\R^4 \to \R$ contains a saddle-center equilibrium point lying in a strictly convex sphere-like singular subset $S_0\subset H^{-1}(0)$. From previous work \cite{PS} we know that for any small energy $E>0$, the energy level $H^{-1}(E)$ contains a closed $3$-ball $S_E$ in a neighborhood of $S_0$ admitting a singular foliation called $2-3$ foliation. One of the binding orbits of this singular foliation is the Lyapunoff orbit $P_{2,E}$ contained in the center manifold of the saddle-center. The other binding orbit lies in the interior of $S_E$ and spans a one parameter family of disks transverse to the Hamiltonian vector field. In this article we show that the $2-3$ foliation forces the existence of infinitely many periodic orbits and infinitely many homoclinics to $P_{2,E}$ in $S_E$. Moreover, if the branches of the stable and unstable manifolds of $P_{2,E}$ inside $S_E$ do not coincide then the Hamiltonian flow on $S_E$ has positive topological entropy. We also present applications of these results to some classical Hamiltonian systems.
\end{abstract}

%
%In \cite{PS} we have found $2-3$ foliations near a critical energy level of a Hamiltonian function $H:\R^4 \to \R$. The critical level is assumed to contain a saddle-center equilibrium point lying in a strictly convex singular subset $S_0\subset H^{-1}(0)$. Each $2-3$ foliation is defined on a closed $3$-ball $S_E\subset H^{-1}(E)$ near $S_0$ and one of its binding orbits is the hyperbolic orbit $P_{2,E}\subset \partial S_E$ sitting in the center manifold of the saddle-center. In the present article we assume that the Hamiltonian function $H$ is real-analytic and use the $2-3$ foliation established in \cite{PS} to show the existence of infinitely many periodic orbits and infinitely many homoclinics to $P_{2,E}$ on $S_E$.

%As an example, both Poincar\'e and Birkhoff studied the circular planar restricted three body problem with focus in questions related to the existence and the multiplicity of periodic orbits.
\maketitle

\tableofcontents

\section{Introduction}
Various phenomena in nature can be modeled by two-degree-of-freedom Hamiltonian systems. Due to the preservation of energy, one usually restricts the study of the dynamics to a fixed $3$-dimensional energy level. Global surfaces of sections, when they exist, provide additional reduction of the flow to an area preserving surface map. When the global surfaces of section are not available, one may still search for the so called systems of transversal sections as introduced by Hofer, Wysocki and Zehnder in \cite{fols}. These systems are singular foliations of the energy level whose singular set is formed by finitely many periodic orbits called bindings. The regular leaves are punctured embedded surfaces foliating the complement of the bindings. They are transverse to the Hamiltonian vector field and are asymptotic to the bindings at the punctures. Systems of transversal sections  may determine transition maps between some regular leaves and valuable information about the dynamics may be extracted from standard tools in surface discrete dynamics.

In this paper we study real-analytic Hamiltonian systems in $\R^4$ admitting a special type of system of transversal sections, the so called $2-3$ foliation. This is a singular foliation of a closed $3$-ball $B$ with precisely two bindings. One of the bindings is hyperbolic and lies in the boundary $\partial B$. It separates $\partial B$ into two hemispheres which are regular leaves of the foliation. The other binding lies in the interior of $B$ and spans a  family of disk-like regular leaves parameterized by an open interval. At the ends of this open interval, the family of disks `breaks' into the union of a cylinder connecting the two bindings and one of the hemispheres of $\partial B$, depending on the end of the interval, see Figure \ref{fig_23} below.

A $2-3$ foliation is proved to exist in \cite{PS} for small positive energies. More precisely, the zero energy level of the Hamiltonian function is supposed to contain a strictly convex sphere-like singular subset $S_0$ with a unique singularity corresponding to a saddle-center equilibrium point. Then for every small positive energy, the energy level contains a closed $3$-ball $S_E$ in a neighborhood of $S_0$ and $S_E$ admits a $2-3$ foliation, see \cite[Theorem 1.9]{PS}. The binding  $P_{2,E} \subset \partial S_E$ coincides with the Lyapunoff orbit in the center manifold of the saddle-center. It has Conley-Zehnder index $2$ and represents an obstruction for the existence of a disk-like global surface of section as constructed in \cite{convex}, see also \cite{HS1}.% and the disk-like regular leaves, parametrized by an open interval, intersect branches of its stable and unstable manifolds.

The main result in this article asserts that the $2-3$ foliation obtained in \cite{PS} forces the existence of infinitely many periodic orbits and infinitely many homoclinics to $P_{2,E}$ inside $S_E$. We split the proof in two cases: if the branches of the stable and unstable manifolds of $P_{2,E}$ coincide then part of the dynamics inside $S_E$ is reduced to a homeomorphism of an open punctured disk which preserves a finite area and has infinite twist near the punctures. Infinitely many periodic orbits are thus derived from a generalization of the Poincar\'e-Birkhoff fixed point theorem due to Franks \cite{franks1}. If the invariant branches of the stable and unstable manifolds do not coincide then the real-analyticity of the Hamiltonian is used to find infinitely many transverse homoclinics.  In particular, there exist horseshoe-type subsystems implying positivity of the topological entropy. Finally, we apply our results to a couple of Hamiltonian systems arising in celestial mechanics and nanotechnology.

\section{Main result}
Let $H:\R^4 \to \R$ be a real-analytic function. Consider coordinates $(x_1,x_2,y_1,y_2)$ in $\R^4$ and let
$$
\omega_0= \sum_{i=1,2} dy_i \wedge dx_i
$$ be the standard symplectic form. The Hamiltonian vector field associated to $H$ is defined by
$$
i_{X_H}\omega_0 = -dH,
$$ and so it is given by  $X_H= J_0 \nabla H,$ where $J_0$ is the usual symplectic matrix $$J_0 = \left(\begin{array}{cc} 0 & {\rm I}_{2 \times 2}\\ -{\rm I}_{2 \times 2} & 0\end{array} \right) $$and $\nabla H$ is the gradient of $H$.

Let us assume that the energy level $H^{-1}(0)$ contains a saddle-center equilibrium point $p_c$. This is a critical point of $H$ so that the matrix $J_0 {\rm Hess}H(p_c)$ has a pair of real eigenvalues $\pm \alpha$ and a pair of purely imaginary eigenvalues $\pm \omega i$ for some $\alpha,\omega>0$. Here ${\rm Hess}H$ is the Hessian matrix of $H$.

By a theorem of J. Moser \cite{moser} and a refinement of H. R\"ussmann \cite{Russ} (see also Delatte \cite{del}) there exists a real-analytic symplectic change of coordinates $(x_1,x_2,y_1,y_2)=\varphi(q_1,q_2,p_1,p_2)\in \R^4,$ defined on neighborhoods $V,U\subset \R^4$ of $0,p_c$, respectively, so that
$$
\varphi^*\omega_0 = \sum_{i=1,2} dp_i \wedge dq_i,
$$ and, up to adding a constant to $H$ and multiplying it by $-1$, the Hamiltonian function $\varphi^*H$ takes the form
\begin{equation}\label{eq_normalform} K(q_1,q_2,p_1,p_2) = \bar K(I_1,I_2) = -\alpha I_1 + \omega I_2 + R(I_1,I_2),\end{equation}
where $I_1=q_1p_1$, $I_2=\frac{q_2^2+p_2^2}{2}$ and $R(I_1,I_2)=O(I_1^2+I_2^2)$.
According to our conventions, this last identity means that there exists $C>0$ so that $|R(I_1,I_2)| \leq C (I_1^2+I_2^2)$ on a small neighborhood of $(0,0)$.
The local coordinates $(q_1,q_2,p_1,p_2)$ near $p_c$ are called Moser's coordinates.

The trajectories of $\dot z= J_0\nabla K (z)$, $z=(q_1,q_2,p_1,p_2)$,  are given by
\begin{equation}\label{eqK}
\left\{\begin{aligned} & q_1(t)  =q_1(0)e^{ -\bar \alpha t},\\ & p_1(t) =p_1(0)e^{ \bar \alpha t},\\ & q_2(t) + ip_2(t) = (q_2(0) + ip_2(0))e^{-i \bar \omega t}, \end{aligned} \right.
\end{equation}
and project to the planes $(q_1,p_1)$ and $(q_2,p_2)$ as in Figure \ref{fig_selacentro}. Here
$$\begin{aligned} \bar \alpha(I_1,I_2) & := -\partial_{I_1} \bar K(I_1,I_2) = \alpha -
\partial_{I_1} R(I_1,I_2),\\\bar \omega(I_1,I_2) & := \partial_{I_2} \bar K(I_1,I_2) = \omega +
\partial_{I_2} R(I_1,I_2),\end{aligned}$$ are constant along the solutions of \eqref{eqK} since $I_1$ and $I_2$ are first integrals of the flow.

\begin{figure}[h!!]
  \centering
  \includegraphics[width=0.7\textwidth]{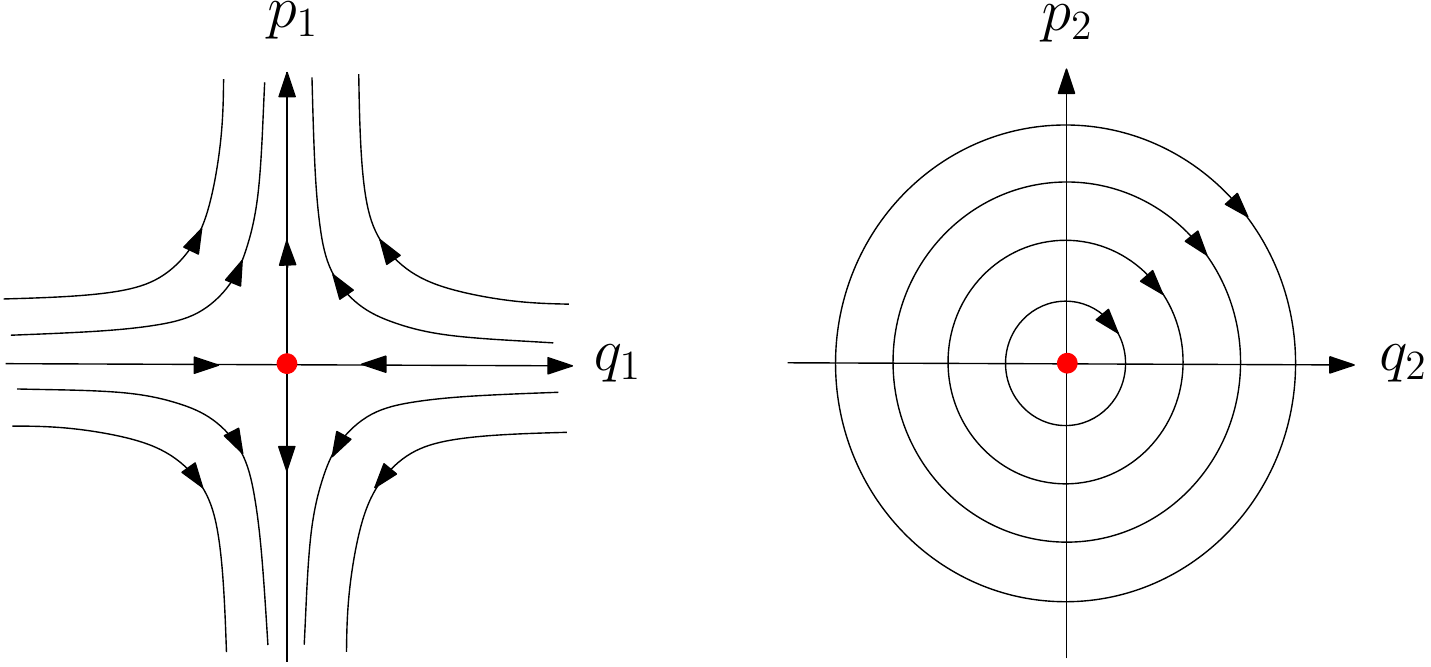}
  \caption{Local behavior of the flow near a saddle-center equilibrium point projected to the planes $(q_1, p_1)$ and $(q_2, p_2)$, respectively.}
  \label{fig_selacentro}
\hfill
\end{figure}

The projection of the energy levels to the $(q_1,p_1)$-plane is depicted in Figure \ref{fig_projecao}. The critical energy level $K^{-1}(0)$ in these local coordinates projects to the first and third quadrants of the plane $(q_1, p_1)$. Without loss of generality, we focus on the subset of the critical level projecting to the first quadrant.
\begin{figure}[ht!!]
  \centering
  \includegraphics[width=0.78\textwidth]{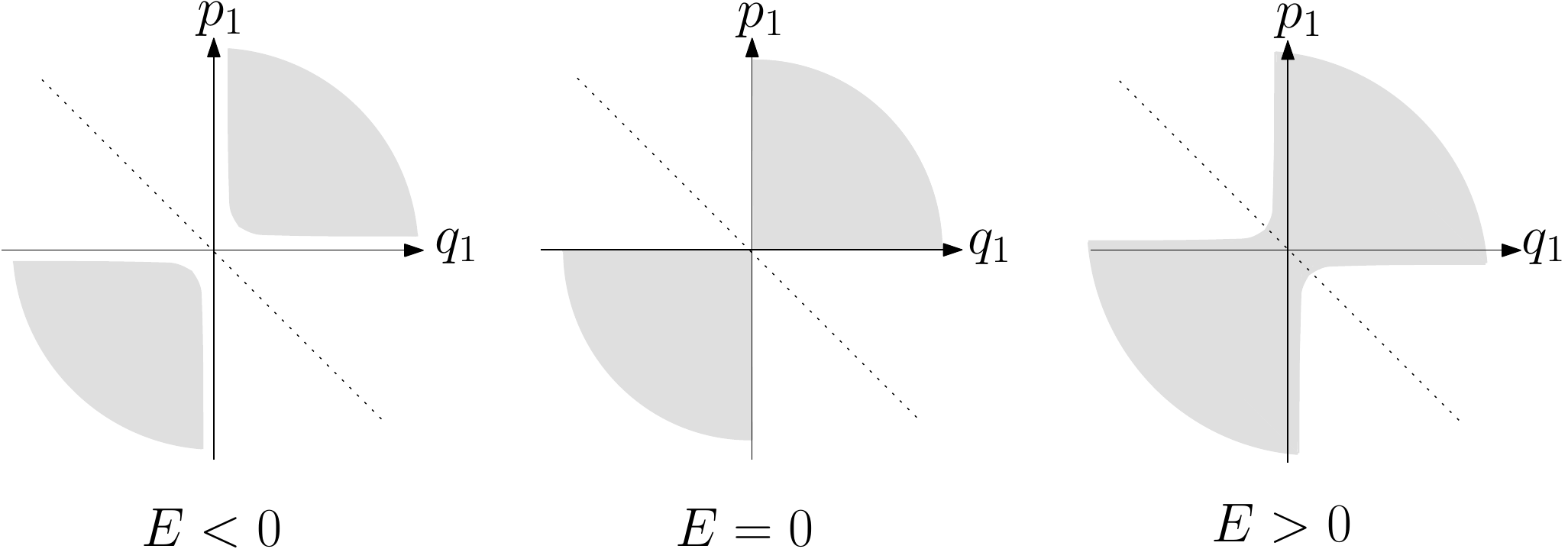}
  \caption{The projections of $K^{-1}(E)$ to the plane $(q_1,p_1)$ for $E<0$, $E=0$ and $E>0$, respectively.  }
  \label{fig_projecao}
\hfill
\end{figure}
For all $\delta>0$ small we let $N_0^\delta \subset K^{-1}(0)$ be the embedded $2$-sphere \cite[Lemma 1.4]{PS} defined by $$N_0^\delta:=K^{-1}(0)\cap \{q_1+p_1 =\delta\}.$$ Then $N_0^\delta$ bounds a topological closed $3$-ball $B_0^\delta \subset K^{-1}(0)$ given  by $$B_0^\delta:=\bigcup_{0\leq \epsilon \leq \delta} N_0^{\epsilon}=K^{-1}(0) \cap \{0\leq q_1+p_1 \leq\delta\}.$$ Notice that $B_0^\delta$ contains the saddle-center equilibrium as its center.

Our first global assumption on the Hamiltonian function $H$ is that  $\varphi(N_0^\delta)$ is also the boundary of an embedded closed $3$-ball $B_\delta \subset H^{-1}(0)$ which contains only regular points of $H$. We obtain what we call a sphere-like singular subset $$S_0:= \varphi(B_0^\delta) \cup B_\delta \subset H^{-1}(0),$$ which is homeomorphic to a $3$-sphere and contains only regular points of $H$, except for the saddle-center $p_c$ as its unique singularity. See Figure \ref{fig_nreg}.

\begin{figure}[ht!!]
  \centering
  \includegraphics[width=0.36\textwidth]{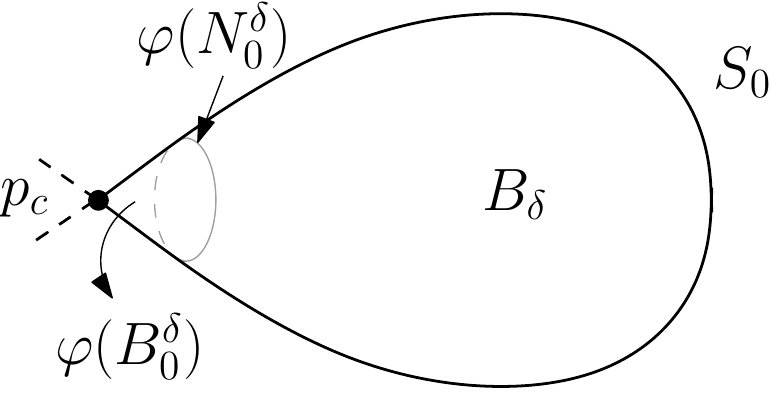}
  \caption{The sphere-like singular subset $S_0=\varphi(B_0^\delta) \cup B_\delta \subset H^{-1}(0).$ }
  \label{fig_nreg}
\hfill
\end{figure}

It follows from the assumptions on $S_0$ that, for all energies $E>0$ sufficiently small, the energy level $H^{-1}(E)$ contains a subset $S_E$ close to $S_0$ which is diffeomorphic to the closed $3$-ball and whose boundary is given in local coordinates by $$\varphi^{-1}(\partial S_E) = K^{-1}(E) \cap \{q_1+p_1 = 0\}.$$

From now on we shall identify the points $U \simeq V$ under the change of coordinates $\varphi$ in order to simplify the notation.  Locally, the projection of $S_E$ to the $(q_1,p_1)$-plane lies in $\{q_1+p_1\geq 0\}$. Notice that $\partial S_E$ contains the periodic orbit $$P_{2,E}:=K^{-1}(E) \cap \{q_1=p_1=0\},$$ called Lyapunoff orbit, which lies in the center manifold of $p_c$. See Figure \ref{fig_se}. It is well known that $P_{2,E}$ is hyperbolic inside its energy level $H^{-1}(E)$ and its Conley-Zehnder index equals $2$, see \cite[Proposition 4.5]{PS}.

\begin{figure}[ht!!]
  \centering
  \includegraphics[width=0.42\textwidth]{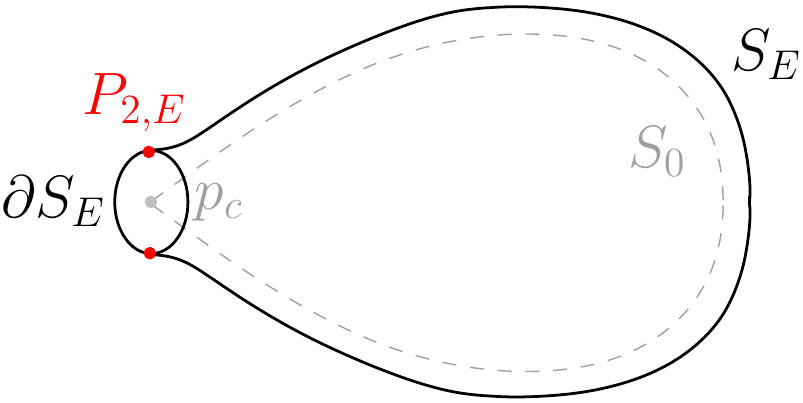}
  \caption{The embedded closed $3$-ball $S_E \subset H^{-1}(E),$ $E>0$ small.}
  \label{fig_se}
\hfill
\end{figure}

Since $\partial_{I_2} \bar K(0,0) = \omega \neq 0$, we can write $$I_2=I_2(I_1,E)=\frac{E}{\omega} +\frac{\alpha}{\omega}I_1 + O(I_1^2+E^2),$$ for $|E|,|I_1|$ sufficiently small with $I_2(0,0)=0$. Thus the periodic orbit $P_{2,E}$ is given by $$q_1=p_1=0, \ \ q_2^2+p_2^2 =2I_2(0,E)=2\frac{E}{\omega}+O(E^2).$$ The Hamiltonian period of $P_{2,E}$ is given by $$T^H_{2,E}= \frac{2\pi}{\bar \omega(0,I_2(0,E))},$$ where $\bar \omega(0,I_2(0,E)) = \partial_{I_2}\bar K(0,I_2(0,E)) = \omega + O(E).$

We are interested in studying the Hamiltonian dynamics on the $3$-ball $S_E\subset H^{-1}(E)$.  Consider the hemispheres of $\partial S_E$ given by $$\begin{aligned}U_{1,E} & = K^{-1}(E) \cap \{q_1+p_1=0,q_1 < 0\},\\ U_{2,E} & = K^{-1}(E) \cap \{q_1+p_1=0,q_1 > 0\},\end{aligned}$$ so that $\partial S_E = U_{1,E} \cup P_{2,E} \cup U_{2,E}.$ See Figure \ref{esfera}.

\begin{figure}[ht!!]
  \centering
  \includegraphics[width=0.75\textwidth]{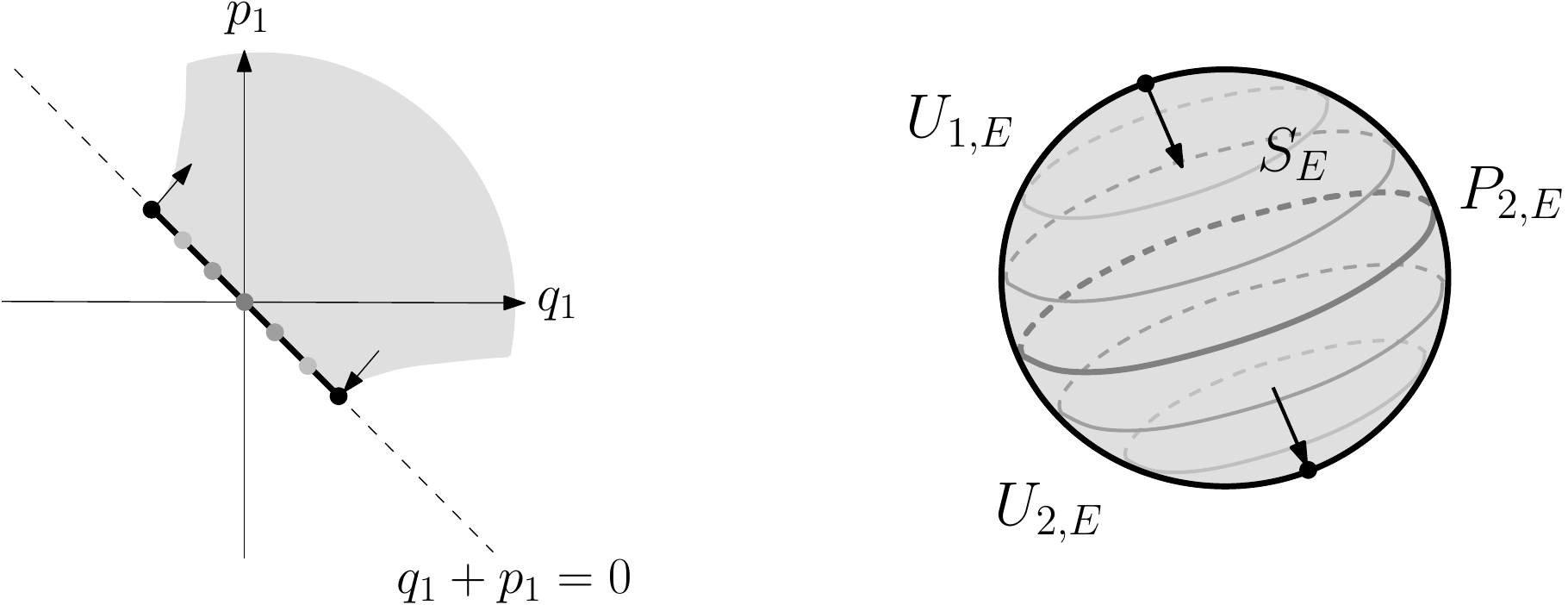}
  \caption{The embedded $2$-sphere $\partial S_E = U_{1,E} \cup P_{2,E} \cup U_{2,E} $ and its projection to the plane $(q_1,p_1)$.}
  \label{esfera}
\hfill
\end{figure}

The stable and unstable manifolds of $P_{2,E}$ are denoted by $W^s(P_{2,E})$ and $W^u(P_{2,E})$, respectively. Locally, such manifolds are given by $$\begin{aligned}W^s_{\rm loc}(P_{2,E}) := \{p_1=0, q_2^2+p_2^2 = 2I_2(0,E)\},\\  W^u_{\rm loc}(P_{2,E}):= \{q_1=0, q_2^2+p_2^2 = 2I_2(0,E)\}. \end{aligned}$$ Let us consider only the branches of $W^s_{\rm loc}(P_{2,E})$ and $W^u_{\rm loc}(P_{2,E})$ contained in $\dot S_E:= S_E \setminus \partial S_E$. %Our aim is to find conditions on $S_0$ so that these branches intersect at infinitely many homoclinics to $P_{2,E}$.

\begin{definition}A $2-3$ foliation of $S_E$ adapted to the Hamiltonian flow is a singular foliation $\F_E$ of $S_E\subset H^{-1}(E)$ so that:
\begin{itemize}
\item[(i)] The singular set of $\F_E$ is formed by $P_{2,E} \subset \partial S_E$ and a periodic orbit $P_{3,E}\subset \dot S_E$, which is unknotted and has Conley-Zehnder index $3$.
\item[(ii)]  $\F_E$ contains the hemispheres $U_{1,E}$ and $U_{2,E}$ in $\partial S_E$  as regular leaves. These leaves are called rigid planes.
\item[(iii)] $\F_E$ contains a cylinder $V_E$ in $\dot S_E$ whose closure has boundary $P_{3,E}\cup P_{2,E}$. $V_E$ is called a rigid cylinder.
\item[(iv)] $\F_E$ contains a one parameter family of planes $D_{\tau,E}, \tau \in (0,1),$ foliating $\dot S_E \setminus (P_{3,E} \cup V_E)$. The closure of each $D_{\tau,E}$ has boundary  $P_{3,E}$. As $\tau \to 0^+$, $D_{\tau,E}$ $C^0$-converges  to $V_E \cup P_{2,E} \cup U_{1,E}$ and, as $\tau \to 1^-$, $D_{\tau,E}$ $C^0$-converges  to $V_E \cup P_{2,E} \cup U_{2,E}$.
\item[(v)] All regular leaves $U_{1,E},U_{2,E},V_E$ and $D_{\tau,E},\tau\in (0,1)$, are transverse to the Hamiltonian vector field $X_H$.
\end{itemize} See Figure \ref{fig_23}.
\end{definition}

\begin{figure}[h!!]
  \centering
  \includegraphics[width=0.5\textwidth]{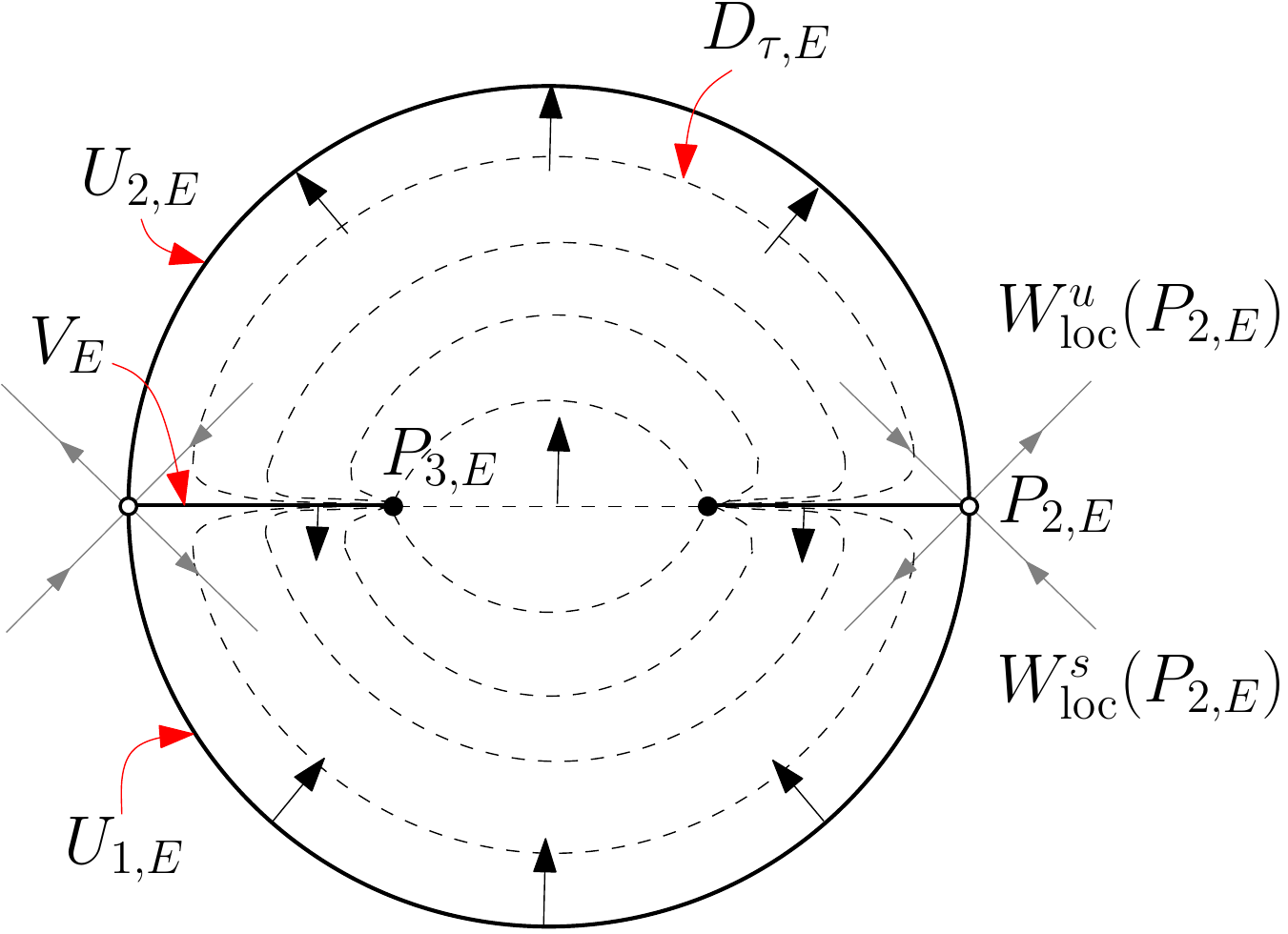}
  \caption{A section of a $2-3$ foliation of $S_E$. The black arrows point to the same direction of the Hamiltonian vector field.}
  \label{fig_23}
\hfill
\end{figure}

\begin{definition}We say that the sphere-like subset $S_0\subset H^{-1}(0)$ containing the saddle-center $p_c$ is a strictly convex singular subset of $H^{-1}(0)$ if $S_0$ bounds a convex subset of $\R^4$ and the Hessian of $H$ restricted to $T\dot S_0$ is everywhere positive-definite. Here $\dot S_0 = S_0 \setminus \{p_c\}$.  \end{definition}

It is proved in \cite{PS} that if $S_0$ is a strictly convex singular subset of $H^{-1}(0)$ then $S_E$ admits a $2-3$ foliation for all $E>0$ sufficiently small.

\begin{theo}\cite[Theorem 1.9 and Corollary 1.10]{PS}\label{teo_PS}
Let $H:\R^4 \to \R$ be a smooth Hamiltonian function. Assume that $H$ admits a saddle-center equilibrium point $p_c\in H^{-1}(0)$ so that, near $p_c$,  $H$ is real-analytic and takes the form \eqref{eq_normalform} for suitable local real-analytic symplectic coordinates.  Assume moreover that $p_c$  lies in a strictly convex singular subset $S_0\subset H^{-1}(0)$. Then, for all $E>0$ sufficiently small, the embedded closed $3$-ball $S_E\subset H^{-1}(E)$ near $S_0$, defined as above, admits a $2-3$ foliation adapted to the Hamiltonian flow. In particular, $P_{2,E}$ admits a homoclinic orbit inside $S_E\setminus \partial S_E$.
\end{theo}

The main result of this paper is the following theorem which asserts that if the Hamiltonian function in Theorem \ref{teo_PS} is real-analytic then one obtains infinitely many periodic orbits and homoclinics to $P_{2,E}$ inside the closed $3$-ball $S_E$.

\begin{theo}\label{maintheo}
Let $H:\R^4 \to \R$ be a real-analytic Hamiltonian function. Assume that $H$ admits a saddle-center equilibrium point $p_c\in H^{-1}(0)$ so that  $H$ takes the form \eqref{eq_normalform} for suitable local real-analytic symplectic coordinates near $p_c$.  Assume moreover that $p_c$  lies in a strictly convex singular subset $S_0\subset H^{-1}(0)$. Let $S_E\subset H^{-1}(E), E>0$ small, be the topological closed $3$-ball near $S_0$, defined as above. Then for all $E>0$ sufficiently small there exist infinitely many periodic orbits and infinitely many homoclinics to $P_{2,E}$ inside $S_E \setminus \partial S_E$. Moreover, if the branches of $W^s(P_{2,E})$ and $W^u(P_{2,E})$ contained in $S_E\setminus \partial S_E$ do not coincide then the Hamiltonian flow restricted to $H^{-1}(E)$ has positive topological entropy. \end{theo}

\section{Applications}

Hamiltonian functions admitting saddle-center equilibrium points are easily found in literature. For instance, if the Hamiltonian $H: \R^4 \to \R$ is written as kinetic plus potential energy
\begin{equation}\label{ecep}
H(x,y,p_x,p_y)=\frac{p_x^2+p_y^2}{2}+V(x,y),
\end{equation}
then $(x_c,y_c) \in \R^2$ is a saddle-type critical point for $V$ if and only if $p_c=(x_c,y_c, 0, 0) \in \R^4$ is a saddle-center equilibrium point for $H$.

In case $H$ has the form \eqref{ecep}, the convexity condition found in Theorems \ref{teo_PS} and \ref{maintheo} can be checked in terms of the potential function $V$ on the corresponding Hill's region.

\begin{prop}[\cite{Sa1}, Theorem 1]\label{teo_convexidade}
Let $H:\R^4\to \R$ be a Hamiltonian function given as in \eqref{ecep}, where $V:\R^2\to \R$ is a smooth function. Let $S\subset H^{-1}(E_0)$ be a sphere-like singular subset of the level set $H^{-1}(E_0)$, with a singularity $p_c$ corresponding to a saddle-center equilibrium point. Let $\pi:\R^4 \to \R^2$ be the projection $\pi(x,y,p_x,p_y)=(x,y)$ and let $B:=\pi(S)$. Then $S$ is a strictly convex singular subset of $H^{-1}(E_0)$ if and only if
\begin{equation}\label{positivo}2(E_0-V)(V_{xx}V_{yy}-V_{xy}^2)+V_{xx}V_y^2+V_{yy}V_x^2-2V_xV_yV_{xy}>0 \end{equation}
for all points in $\pi(S \setminus \{p_c\}) = B\setminus \{\pi(p_c)\}$. This statement also holds if $S$ is diffeomorphic to the $3$-sphere and inequality \eqref{positivo} holds at the disk-like region $B=\pi(S)$.
\end{prop}

If the Hamiltonian $H: \R^4 \to \R$ has the form
\begin{equation}
H(x,y,p_x,p_y)=\frac{1}{2}\left[(p_x-A_1(x,y))^2+(p_y-A_2(x,y))^2\right]+V(x,y),
\end{equation}
where the magnetic vector potential $A=(A_1,A_2)$ is linear,  then strict convexity of the singular  subset does not depend on $A$ and can  be directly checked using inequality \eqref{positivo} on the corresponding Hill's region $\pi(S \setminus \{p_c\})$. %In \cite{AB1}, Asselle and Benedetti study multiplicity of closed trajectories of magnetic flows on closed surfaces.

In the following sections we present some examples of Hamiltonians of the form kinetic plus potential energy for which our main results apply. In particular, such systems admit infinitely many periodic orbits and homoclinics to the Lyapunoff orbit near the critical energy level containing a saddle-center equilibrium. Motivated by the Allen-Cahn equation, Fusco-Gronchi-Novaga \cite{FGN1,FGN2} use variational methods to study the existence of periodic motions  of kinetic plus potential Hamiltonians.

\subsection{Buckled nanobeams}

The transverse displacement of a nanobeam subjected to a longitudinal compressive stress applied at both ends is studied in \cite{chakra1, CEW}. Under a certain compression, the linearized dynamics over the equilibrium state can be formulated in terms of a two-degree-of-freedom Hamiltonian system, which is obtained by a two-mode Galerkin truncation of the infinite dimensional dynamics described by the Euler-Bernoulli beam equation. The Hamiltonian function describing such a truncated dynamics is given by
\begin{equation}\label{ham_modal_simp}
H(x,y,p_x,p_y)=\frac{p_x^2+p_y^2}{2} + V(x,y),
\end{equation}
where the potential function $V$ has the form
\begin{equation}\label{pot_modal_simp}
V(x,y)=\alpha x^2+4\beta y^2+\frac{1}{2}\left(x^2+4y^2\right)^2
\end{equation}
with real parameters $\alpha<0$ and $\beta \neq 0$. Notice that $V$ is symmetric with respect to $x$ and $y$.

Let us assume that $\beta>0$. In this case, $H$ admits a saddle-center equilibrium  at $p_c:=(0,0,0,0)\in H^{-1}(0)$ corresponding to a saddle of the potential function $V$.  The critical energy level $H^{-1}(0)$ contains a pair of singular subsets $S$ and $S'$, both homeomorphic to $S^3$, intersecting at the common singularity $p_c$. Their projections $B$ and $B'$ to the plane $(x,y)$, respectively, are depicted in Figure \ref{hamnano}.
\begin{figure}[h]
\centering
\includegraphics[scale=.6]{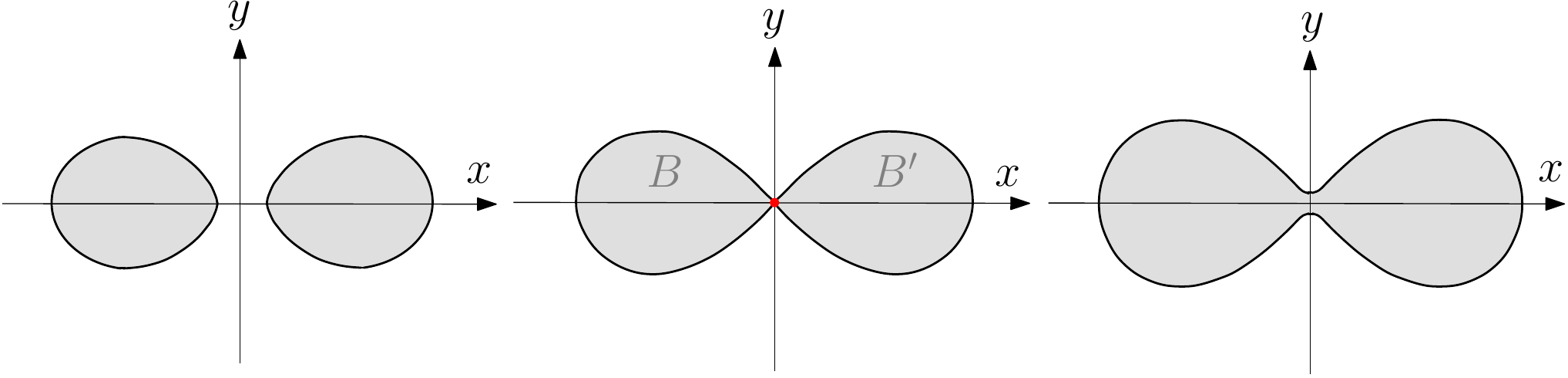}
\caption{Hill's regions of the Hamiltonian function defined by \eqref{ham_modal_simp} and \eqref{pot_modal_simp} for energies $E<0$, $E=0$ and $E>0$, respectively, for $|E|$ small. $S$ and $S'$ project to $B$ and $B'$, respectively.}
\label{hamnano}
\end{figure}

\begin{prop}\label{prop_convexity_nano}
Assume that $\alpha<0<\beta$. Then the sphere-like singular subsets $S$ and $S'$ are strictly convex.	
\end{prop}

\begin{proof}
	Consider the disks $B=\pi(S)$ and $B'=\pi(S')$, where $\pi:\R^4 \to \R^2$ is the canonical projection $\pi(x,y,p_x,p_y)=(x,y)$.  From  \eqref{ham_modal_simp}
	we see that $V\leq 0$ for all points in $B\cup B' \subset \pi(H^{-1}(0))$, i.e.,
	\begin{equation}\label{ineqVbarra}
	\alpha x^2+4\beta y^2+\frac{1}{2}\left(x^2+4y^2\right)^2 \leq 0, \,\,\,\, \forall (x,y) \in B\cup B'.
	\end{equation}
	
	According to Proposition \ref{teo_convexidade}, a necessary and sufficient condition for $S$ and $S'$ to be strictly convex singular subsets is that
	\begin{equation}\label{t}
	T:=-2V \det {\rm Hess} V+V_{xx}V_y^2+V_{yy}V_x^2-2V_xV_yV_{xy}>0
	\end{equation}
	for all points in $D:=B \cup B'\setminus \{(0,0)\}$. A straightforward computation shows that
	$$T=-16\left(x^2+4y^2\right)^2 \left[\left(x^2+4y^2\right)^2+\alpha\left(3\beta+3x^2+4y^2\right)+\beta\left(x^2+12y^2\right)\right],$$
	and hence we just need to prove that
	$$C:=\left(x^2+4y^2\right)^2+\alpha\left(3\beta+3x^2+4y^2\right)+\beta\left(x^2+12y^2\right)<0$$
	for all $(x,y) \in D$.
	
	Using \eqref{ineqVbarra} one can check that
	\begin{equation}\label{C}
	C \leq 3\alpha \beta+(\alpha+\beta)(x^2+4y^2)
	\end{equation}
	in $D$.
	If $\alpha+\beta \leq 0$ then inequality \eqref{C} directly implies that $C<0$ in $D$ since $\alpha<0<\beta$. Let us assume that $\alpha+\beta>0$.
	In this case,
	$$3\alpha \beta+(\alpha+\beta)(x^2+4y^2)<0  \Leftrightarrow x^2+4y^2 < -\frac{3\alpha \beta}{\alpha+\beta}$$
	and therefore, in order to prove that the function $C=C(x,y)$ is negative on $D$,  we need to check that $D$ is contained in the open subset bounded by the ellipse $$x^2+4y^2=-\frac{3\alpha \beta}{\alpha+\beta}.$$ It is sufficient to show  this condition for points in the boundary $\partial D$.
	
Let $U:\R^2 \to \R$ be given by $U(x,y):=x^2+4y^2$. Equality in \eqref{ineqVbarra} holds if and only if $(x,y) \in \partial D$. This means that
	\begin{equation}\label{2grau}
	U^2+2\alpha U+8(\beta-\alpha)y^2=0
	\end{equation}
	for all points in $\partial D$. The solutions of the polynomial \eqref{2grau} in $U$ are given by
	$$U=\frac{-2\alpha\pm\sqrt{4\alpha^2-32(\beta-\alpha)y^2}}{2}.$$
	Thus, the maximum value of $U$ on the boundary $\partial D$ is assumed when $y=0$, for which $U=-2\alpha >0$ and $x=\pm\sqrt{-2\alpha}$. Using again that $\alpha+\beta>0$, we see that $U(\pm\sqrt{-2\alpha},0)=-2\alpha < -\frac{3\alpha \beta}{\alpha+\beta}$
	is equivalent to $2\alpha^2 > \alpha\beta$, which clearly holds since $\alpha<0<\beta$. This proves that
	\begin{equation}\label{sup}
	 U(x,y)<-\frac{3\alpha \beta}{\alpha+\beta}, \ \ \forall (x,y) \in\partial D,
	\end{equation} as desired, and hence $C<0$ on $D$ in case $\alpha + \beta>0$ as well.
	
	We conclude that $T=T(x,y)$, given by \eqref{t}, is positive on $D$. Theorem \ref{teo_convexidade} implies that the subsets $S, S' \subset H^{-1}(0)$ containing the saddle-center $p_c=(0,0,0,0)$ are both strictly convex.
	\end{proof}

For all $E>0$ sufficiently small, the energy level $H^{-1}(E)$ contains closed $3$-balls $S_E$ and $S_E'$ near $S$ and $S'$, respectively, so that $\partial S_E=\partial S_E'$ and $W_E:=S_E \cup S_E'$ is an embedded $3$-sphere. The projection of $W_E\subset H^{-1}(E)$ to the plane $(x,y)$ is represented in the right-most drawing in Figure \ref{hamnano}.

It follows from Proposition \ref{prop_convexity_nano} and Theorem \ref{teo_PS} that both $S_E$ and $S_E'$ admit $2-3$ foliations, denoted by $\F_E$ and $\F_E'$, respectively, such that the Lyapunoff orbit $P_{2,E} \subset \partial S_E=\partial S_E'$ is a binding orbit for both $\F_E$ and $\F_E'$. In this case, the singular foliation $\F_E \cup \F_E'$ is what we call a $3-2-3$ foliation adapted to Hamiltonian flow on $W_E$, see Figure \ref{fig_323} and \cite[Remark 1.11]{PS}. Theorem \ref{maintheo} ensures the existence of infinitely many periodic orbits and infinitely many homoclinics to $P_{2,E}$ in each subset $S_E \setminus \partial S_E$ and $S_E'\setminus \partial S_E'$. %If the branches of $W^s(P_{2,E})$ and $W^u(P_{2,E})$ contained in $S_E\setminus \partial S_E$ (or equivalently in $S_E'\setminus \partial S_E'$) do not coincide, then the Hamiltonian flow restricted to the energy level $H^{-1}(E)$ has positive topological entropy.

\begin{theo}For all $E>0$ sufficiently small, $W_E$ admits a $3-2-3$ foliation. In particular, $W_E$ contains infinitely many periodic orbits and infinitely many homoclinics to the Lyapunoff orbit $P_{2,E}$ in the center manifold of the saddle-center $0\in \R^4$.  \end{theo}

\begin{figure}[h!!]
	\centering
	\includegraphics[width=0.6\textwidth]{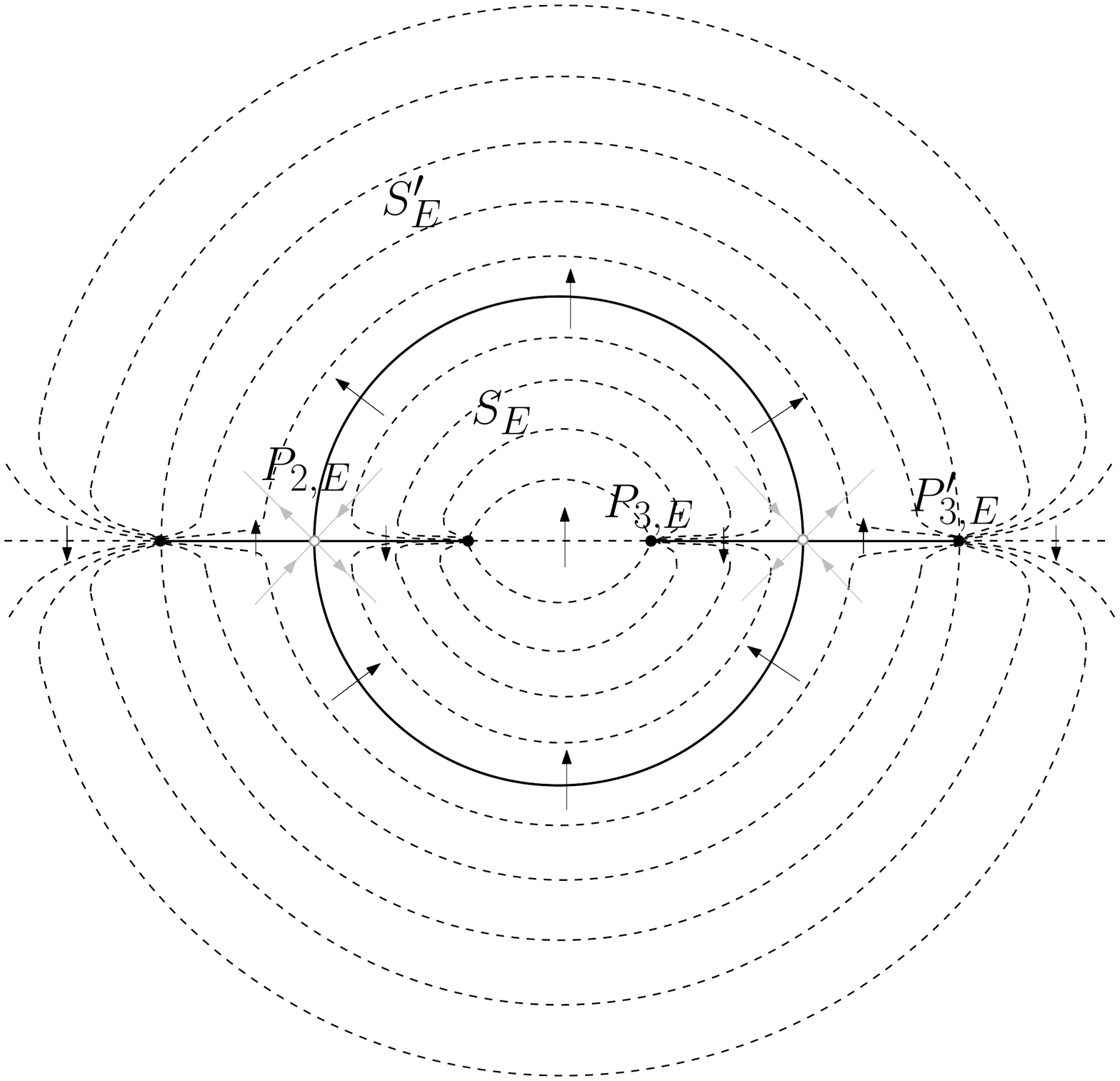}
	\caption{A section of a $3-2-3$ foliation on $W_E=S_E \cup S_E'$. The black arrows point to the same direction of the Hamiltonian vector field.}
	\label{fig_323}
	\hfill
\end{figure}

\subsection{The H\'enon-Heiles Hamiltonian}

Let $H:\R^4 \to \R$ be the Hamiltonian defined by
\begin{equation}\label{eq_hamHH}
H(x,y,p_x,p_y)=\frac{p_x^2+p_y^2}{2}+ \frac{x^2+y^2}{2}+bx^2y-\frac{y^3}{3}
\end{equation}
which depends on a parameter $0<b<1$.
The case $b=1$ is known as H\'enon-Heiles Hamiltonian and it models the motion of stars around a galactic center.

The Hamiltonian $H$ admits a saddle-center equilibrium at $p_c=(0,1,0,0) \in H^{-1}\left(\frac{1}{6}\right)$. Moreover, for all $0<b<1$, $p_c$ lies in a strictly convex singular subset $S_0 \subset H^{-1}\left(\frac{1}{6}\right)$, see \cite{Sa1} for a proof. The projection $B$ of $S_0$ to the $(x,y)$-plane is seen in Figure \ref{fig_hamHH}.

\begin{figure}[h!!]
  \centering
  \includegraphics[width=1\textwidth]{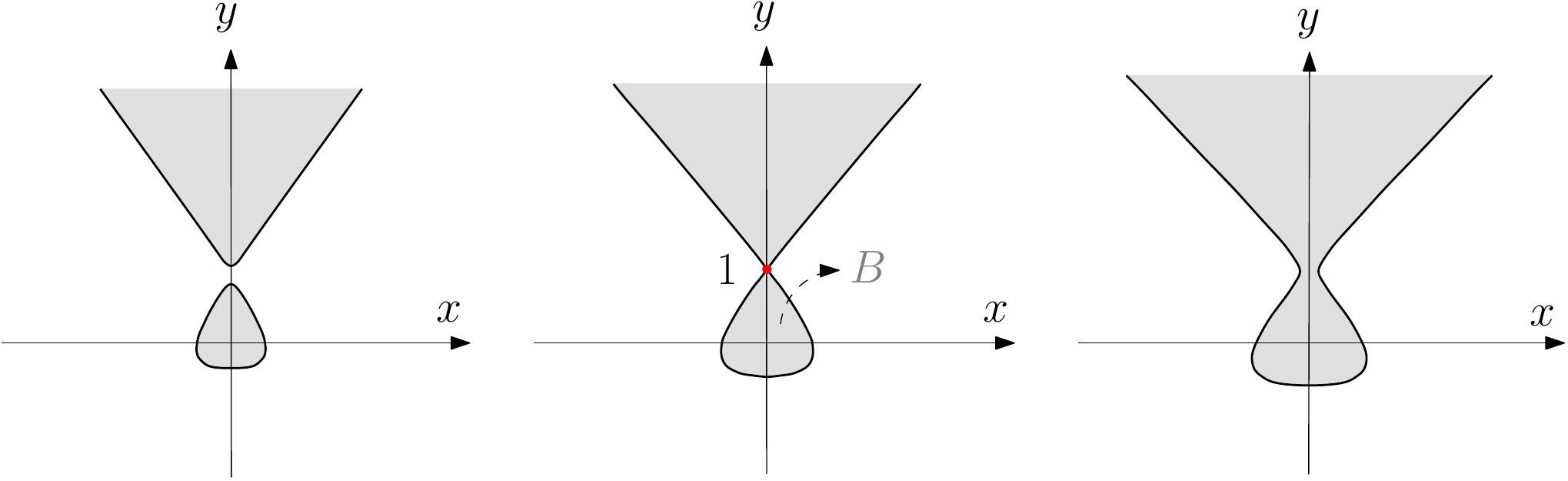}
  \caption{Hill's regions of the Hamiltonian function \eqref{eq_hamHH} for energies $E<\frac{1}{6}$, $E=\frac{1}{6}$ and $E>\frac{1}{6}$ respectively, with $\left|E-\frac{1}{6}\right|$ small.}
	\label{fig_hamHH}
\hfill
\end{figure}

As observed in \cite[\S 1.5]{PS}, Theorem \ref{teo_PS} gives a $2-3$ foliation $\F_E$ defined on an embedded closed $3$-ball $S_E\subset H^{-1}(E)$ near $S_0$ for each $E>\frac{1}{6}$ sufficiently close to $\frac{1}{6}$. One of the binding orbits of $\F_E$ is the Lyapunoff orbit $P_{2,E}\subset \partial S_E$.  Theorem \ref{maintheo} implies the following result.

\begin{theo}For every $E>\frac{1}{6}$ sufficiently close to $\frac{1}{6}$ the Hamiltonian flow on the closed $3$-ball $S_E \subset H^{-1}(E)$ admits infinitely many periodic orbits and infinitely many homoclinics to the Lyapunoff orbit $P_{2,E}$. Moreover, if the branches of the stable and unstable manifolds of $P_{2,E}$ inside $S_E$ do not coincide then the topological entropy of the Hamiltonian flow restricted to $H^{-1}(E)$ is positive.
\end{theo}

\subsection{Charged particles in planetary magnetospheres}

The motion of a dust particle in a planetary magnetosphere, whose dynamics is dominated by gravitational and eletromagnetic forces, is studied in \cite{CPM}, see also \cite{dullin,hhs,mhh,xuwuma}. Under some strong simplifying assumptions, this problem can be reduced to a two-degree-of-freedom Hamiltonian system in $\R^4$ with Hamiltonian function given by

\begin{equation}\label{hamsemresto}
H(x,z,p_x,p_z)=\frac{p_x^2+p_z^2}{2}+V(x,z),
\end{equation} where $V$ is the potential function
\begin{equation}\label{potential}
V(x,z)=-\epsilon x+\frac{a^2}{3}x^3+\frac{c^2}{2}z^2
\end{equation}
depending on the parameters $a,c,\epsilon>0.$

The points
$$p_1=\left(-\frac{\sqrt{\epsilon}}{a},0\right) \hspace{1cm} \mbox{ and } \hspace{1cm} p_2=\left(\frac{\sqrt{\epsilon}}{a},0\right)$$
are critical points of $V$ and hence $P_1:=(p_{1},0,0)$ and $P_2:=(p_{2},0,0)$ are equilibrium points of the Hamiltonian flow of $H$. Their energies are given by
\begin{equation}\label{E1}
E_1=H(P_1)=\frac{2\epsilon\sqrt{\epsilon}}{3a} \hspace{1cm} \mbox{ and } \hspace{1cm} E_2=H(P_2)=-\frac{2\epsilon\sqrt{\epsilon}}{3a}.
\end{equation}
Since $\mbox{Hess}V(x,z)=\mbox{Diag}\left( 2a^2x,c^2\right)$ we see that $P_1$ is a saddle-center equilibrium point of $H$ and $P_2$ corresponds to a local minimum. We shall prove that for any $a,c,\epsilon>0$ the saddle-center $P_1$ lies in a strictly convex singular subset $S_0 \subset H^{-1}(E_1)$.

Let $\pi:\R^4 \to \R^2$ be the canonical projection  \begin{equation}\label{projection1}
\pi(x,z,p_x,p_z)=(x,z).
\end{equation}
The projection $\pi(H^{-1}(E_1))$  has a subset $B$ homeomorphic to the closed disk which contains  $\pi(P_1)=p_1$. The boundary $\partial B$ is a closed curve in $\R^2$, which is regular except at  $p_1$, and its points satisfy $V=E_1$. See Figure \ref{fig_ham_magneto}. Since $V\left(-\sqrt{\epsilon}/a,0\right) = V\left(2\sqrt{\epsilon}/a,0\right)=E_1$ we see from \eqref{potential} that
\begin{equation}\label{projection2}
B\subset \left\{(x,z)\in \R^2 \mid -\frac{\sqrt{\epsilon}}{a} \leq x \leq  2\frac{\sqrt{\epsilon}}{a} \right\}.
\end{equation}

The subset $S_0:=\pi^{-1}(B) \cap H^{-1}(E_1)$ contains the saddle-center $P_1$, is homeomorphic to the $3$-sphere $S^3$ and $\dot S_0:=S_0\setminus \{P_1\}$ is a regular hypersurface of $\R^4$.

\begin{figure}[h!!]
	\centering
	\includegraphics[width=1\textwidth]{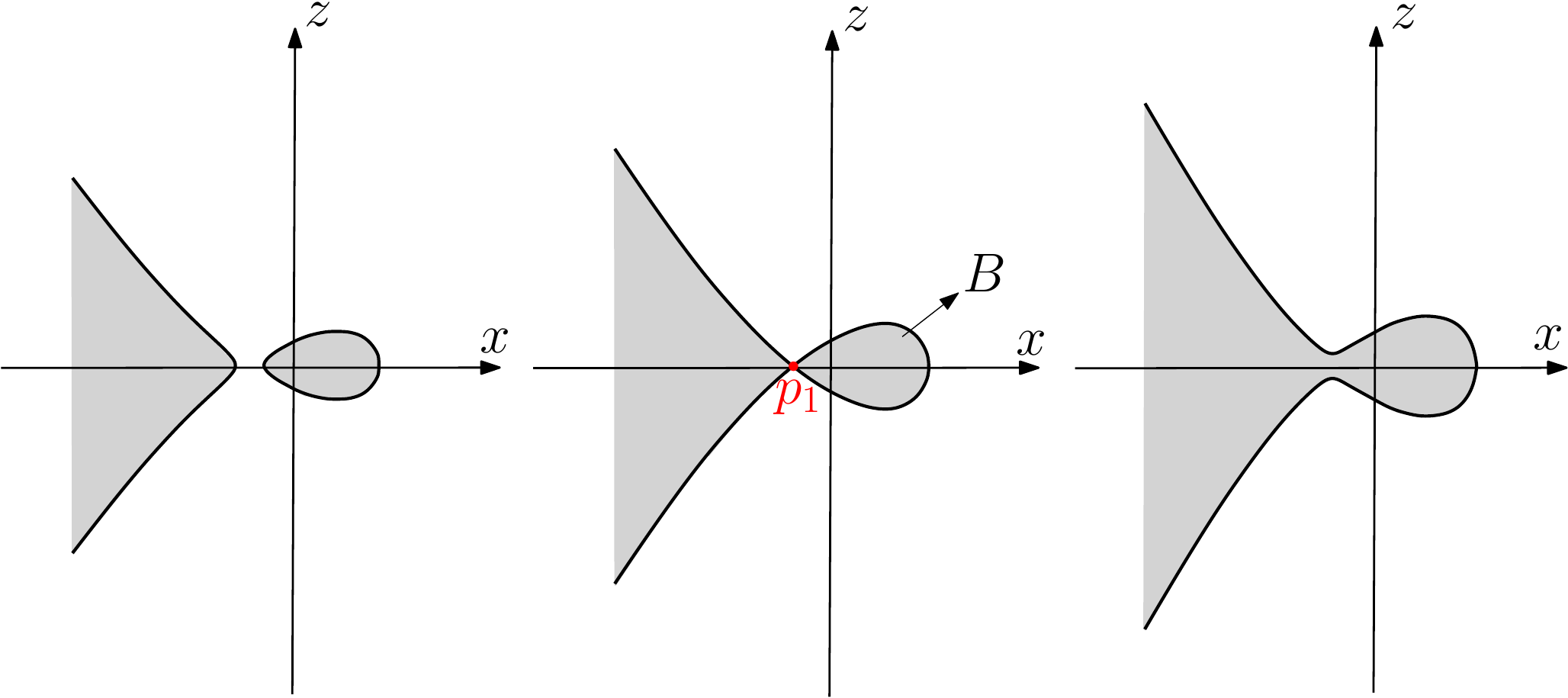}
	\caption{Hill's regions of the Hamiltonian function \eqref{hamsemresto} for energies $E<E_1$, $E=E_1$ and $E>E_1$ respectively, with $|E-E_1|$ small.}
	\label{fig_ham_magneto}
	\hfill
\end{figure}

\begin{prop}\label{prop_magneto}For all $a,c,\epsilon>0$, the sphere-like singular subset $S_0=\pi^{-1}(B) \cap H^{-1}(E_1)$ is strictly convex.
\end{prop}	

\begin{proof}
In order to prove that $S_0$ is strictly convex we need to verify the condition
$$M:=2(E_1-V)\det {\rm Hess}V+V_{xx}V_z^2+V_{zz}V_x^2-2V_xV_zV_{xz}>0,$$
for all points in $\dot B:=B\setminus \{p_1\}$, see Proposition \ref{teo_convexidade}.  A direct computation shows that
$$M(x,z)=\frac{c^2}{3}g(x),$$
where
$$
g(x)=-a^4x^4+6a^2\epsilon x^2+8a\epsilon\sqrt{\epsilon}x+3\epsilon^2.
$$
So it is sufficient to check that $g(x)>0$ for all $x\in \left(-\frac{\sqrt{\epsilon}}{a}, 2\frac{\sqrt{\epsilon}}{a}\right].$

Defining $u:=\frac{a}{\sqrt{\epsilon}} x$, we reduce to proving that $q(u)>0$ on $(-1,2]$, where $$q(u)=-u^4+6u^2+8u+3.$$ One can easily see that $u=-1$ and $u=3$ are roots of $q$ with multiplicities $3$ and $1$, respectively, and thus $q$ is positive on  $(-1, 2]$. Consequently, $M$ is positive  on $\dot B$.
We conclude from Theorem \ref{teo_convexidade} that  $S_0 \subset H^{-1}(E_1)$ is a strictly convex singular subset of $H^{-1}(E_1)$ which contains the saddle-center $P_1$.
\end{proof}

Proposition \ref{prop_magneto} together with Theorems \ref{teo_PS} and \ref{maintheo} imply the following theorem.

\begin{theo}Fix $\epsilon,a,c>0$ and let $E_1$ be as in \eqref{E1}.  For every $E-E_1>0$ sufficiently small the closed $3$-ball $S_E\subset H^{-1}(E)$  admits a $2-3$ foliation adapted to the Hamiltonian flow. Moreover, $S_E$ admits  infinitely many periodic orbits and infinitely many homoclinics to $P_{2,E}$. Since $H$ is integrable, the branches of the stable and unstable manifolds of $P_{2,E}$ inside $S_E$ coincide.
\end{theo}

\section{Existence of a homoclinic to the Lyapunoff orbit}\label{sec_existenciahomoc}

Consider the $2-3$ foliation $\F_E$ on the $3$-ball $S_E \subset H^{-1}(E)$ as the one obtained in Theorem \ref{teo_PS}, for $E>0$ small. The foliation $\F_E$ contains a one parameter family of planes $D_{\tau,E}$, $\tau \in (0,1)$, each one transverse to the Hamiltonian vector field $X_H$, so that the closure of $D_{\tau,E}$ has the periodic orbit $P_{3,E}$ as boundary. We shall study the first return map to such leaves, where it is defined, in order to prove multiplicity of periodic orbits and homoclinics to $P_{2,E}$ in $\dot S_E=S_E \setminus \partial S_E$.
In the following we fix $E>0$  small and assume the existence of the $2-3$ foliation $\F_E$ on $S_E \subset H^{-1}(E)$. From now on we may omit the dependence on $E$ in the notation for simplicity.

For all $0<\tau_0<1$ sufficiently close to $0$, the branch $W^u_{\rm loc}(P_{2,E})$ of the local unstable manifold of $P_{2,E}$ inside $\dot S_E$ intersects the plane $D_{\tau_0}$ on an embedded circle $C^u_{\tau_0}$.  All points in the interior of the closed disk $B^u_{\tau_0}\subset D_{\tau_0}$ bounded by $C^u_{\tau_0}$ correspond to trajectories just entering $S_E$ through the hemisphere $U_{1,E}$.
Similarly, for $0<\tau_1<1$ sufficiently close to $1$,   the branch of the local stable manifold $W^s_{\rm loc}(P_{2,E})$ inside $\dot S_E$ intersects $D_{\tau_1}$ on an embedded circle denoted by $C^s_{\tau_1}$. All points in the interior of the closed disk $B^s_{\tau_1}\subset D_{\tau_1}$ bounded by $C^s_{\tau_1}$ correspond to trajectories which exit $S_E$ through the hemisphere $U_{2,E}$.

Due to the existence of the $2-3$ foliation on $S_E$, the Hamiltonian flow induces the following symplectomorphisms:
\begin{itemize}
\item A global transition map
\begin{equation}\label{psig}\Psi^g: D_{\tau_0} \to D_{\tau_1}\end{equation}
defined as follows: if $x\in D_{\tau_0}$, then $\Psi^g(x)$ is the first point in the positive trajectory through $x$ which hits $D_{\tau_1}$. Such diffeomorphism preserves the canonical symplectic form restricted to the disks $D_{\tau_0}$ and $D_{\tau_1}$.
\item A local transition map
\begin{equation}\label{psil}\Psi^l: D_{\tau_1} \setminus B^s_{\tau_1}\to D_{\tau_0}\setminus B^u_{\tau_0}\end{equation}
defined as follows: if $x\in D_{\tau_1} \setminus B^s_{\tau_1}$, then $\Psi^l(x)$ is the first point in the positive trajectory through $x$ which hits the annulus $D_{\tau_0}\setminus B^u_{\tau_0}$. Such diffeomorphism preserves the canonical symplectic form restricted to the annuli $D_{\tau_1} \setminus B^s_{\tau_1}$ and $D_{\tau_0}\setminus B^u_{\tau_0}$. Later on we shall use Moser's coordinates to describe $\Psi^l$ near the boundary component $C^s_{\tau_1}$.
\end{itemize}
The existence of the global and local maps above follows from the fact that the Conley-Zehnder index of $P_{3,E}$ is $3$ and by a precise description of the asymptotic behavior of the regular leaves $D_\tau, 0<\tau<1,$ close to $P_{3,E}$. See \cite{PS} and \cite{convex}.

Using these transition maps we can prove that the $2-3$ foliation in $S_E$ forces the existence of at least one homoclinic orbit to  $P_{2,E}$ contained in $\dot S_E$.

\begin{prop}\label{prop_homoclinic}The open $3$-ball  $\dot S_E$ contains at least one homoclinic orbit to $P_{2,E}$. \end{prop}

Proposition \ref{prop_homoclinic} follows directly from Proposition \ref{prop_homoclinic2} below. It relies on standard arguments based on preservation of area, see \cite{BGS} and \cite{fols}. We include it here for completeness.

\begin{prop}\label{prop_homoclinic2} Let $\Psi^g$ and $\Psi^l$ be the global and local symplectomorphisms defined in \eqref{psig} and \eqref{psil}, respectively. Let $C^u_{\tau_0}$ and $C^s_{\tau_1}$ be the intersections of $W^u_{\rm loc}(P_{2,E})$ and $W^s_{\rm loc}(P_{2,E})$ with $D_{\tau_0}$ and $D_{\tau_1}$ respectively. Then there exists $N\in \N^*$ so that $(\Psi^g \circ \Psi^l)^{N-1} \circ \Psi^g|_{C^u_{\tau_0}}$ is well-defined and
$$(\Psi^g \circ \Psi^l)^{N-1} \circ \Psi^g(C^u_{\tau_0}) \cap C^s_{\tau_1}\neq \emptyset.$$ \end{prop}

\begin{proof}If $\Psi^g(C^u_{\tau_0}) \cap C^s_{\tau_1}\neq \emptyset$ then $N=1$ and the proof is finished. Otherwise $\Psi^g(C^u_{\tau_0})$ bounds a closed disk $B^u_1:=\Psi^g(B^u_{\tau_0})$ which is necessarily contained in $D_{\tau_1}\setminus B^s_{\tau_1}$. This follows from the fact that $\Psi^g$ preserves an area form and both disks $B^u_{\tau_0}$ and $B^s_{\tau_1}$ have the same area $T_{2,E}$. In this case $\Psi^g \circ \Psi^l|_{B^u_1}$ is well-defined and we define $B^u_2:=\Psi^g \circ \Psi^l(B^u_1)$. If $\partial B^u_2 \cap C^s_{\tau_1}\neq \emptyset$, then $N=2$ and the proof is finished. Otherwise, the closed disk $B^u_2$, which also has symplectic area $T_{2,E}$, is contained in $D_{\tau_1}\setminus B^s_{\tau_1}$ and must be disjoint from $B^u_1$ (note that $\Psi^l(B^u_1)$ is contained in $D_{\tau_0}\setminus B^u_{\tau_0}$). In this case  $\Psi^g \circ \Psi^l|_{B^u_2}$ is well-defined and we define $B^u_3:=\Psi^g \circ \Psi^l(B^u_2)$. Arguing inductively, if $\partial B^u_{j}\cap C^s_{\tau_1} \neq \emptyset$, then $N=j$ and the proof is finished. Otherwise,  the closed disk $B^u_j$, which also has symplectic area $T_{2,E}$, is contained in $D_{\tau_1}\setminus (B^s_{\tau_1} \cup (\cup_{1\leq k< j} B^u_{k}))$. In this case  $\Psi^g \circ \Psi^l|_{B^u_j}$ is well-defined and we define $B^u_{j+1}:=\Psi^g \circ \Psi^l(B^u_j)$. Since the symplectic area of $D_{\tau_1}\setminus B^s_{\tau_1}$ is finite and all $B^u_j$ are disjoint and have the same symplectic area $T_{2,E}>0$, this process has to terminate after finitely many steps, i.e., there exists $N\in \N$ so that $\Psi^g \circ \Psi^l|_{B^u_{N-1}}$ is well-defined and $B^u_N:= \Psi^g \circ \Psi^l(B^u_{N-1})$ intersects $B^s_{\tau_1}$.  In particular, $(\Psi^g \circ \Psi^l)^{N-1} \circ \Psi^g|_{C^u_{\tau_0}}$ is well-defined and $(\Psi^g \circ \Psi^l)^{N-1} \circ \Psi^g(C^u_{\tau_0}) \cap C^s_{\tau_1}\neq \emptyset,$ finishing the proof of the proposition. \end{proof}

Let $N\in\N$ be as in Proposition \ref{prop_homoclinic2}. From now on we use the notations
$$\begin{aligned}
B^u_N & := (\Psi^g \circ \Psi^l)^{N-1} \circ \Psi^g(B^u_{\tau_0}) \subset D_{\tau_1}\\
C^u_N & := \partial B^u_N=(\Psi^g \circ \Psi^l)^{N-1} \circ \Psi^g(C^u_{\tau_0}) \subset D_{\tau_1}.
\end{aligned}$$
 We know that $C^u_N \cap C_{\tau_1}^s \neq \emptyset$ and we consider two distinct situations: either $C^u_N\neq C_{\tau_1}^s$ or $C^u_N=C_{\tau_1}^s$.

In the first case where $C^u_N\neq C_{\tau_1}^s$, we shall prove that the Hamiltonian flow on $\dot S_E$ admits infinitely many transverse homoclinic orbits to $P_{2,E}$ and such transversality implies positivity of the topological entropy. In particular, we obtain infinitely many periodic orbits in $\dot S_E$.  This is discussed in Section \ref{sec_naocoincide}.

In the second case where $C^u_N=C_{\tau_1}^s$, we immediately obtain infinitely many homoclinic orbits to $P_{2,E}$ in $\dot S_E$. Moreover,
the Hamiltonian flow defines a symplectomorphism $\Psi^N: \A \to \A$, where
 $\Psi:=\Psi^g \circ \Psi^l$ and $\A$ consists of $D_{\tau_1}$  with finitely many disjoint closed disks removed
$$\A:= D_{\tau_1} \setminus \bigcup_{j=0}^{N-1} \Psi^{-j}\left(B^s_{\tau_1}\right).$$
Such map  describes the dynamics on an invariant open subset  $\U_\A \subset \dot S_E$ given by the trajectories in $\dot S_E$ intersecting $\A$. Moreover, $\Psi$ has an infinite twist near the inner boundary components of $\A$. Using results on area preserving homeomorphisms of the open annulus due to J. Franks, we shall derive infinitely many periodic points of $\Psi$, which correspond to infinitely many periodic orbits in $\dot S_E$. This is proved in Section \ref{sec_coincide}.

\section{Infinite twist of the local transition map} \label{sec_twist}

In this section we construct local models for neighborhoods of $C^s_{\tau_1} \subset D_{\tau_1}$ and $C^u_{\tau_0}\subset D_{\tau_0}$, for a fixed $E>0$ sufficiently small. We find suitable coordinates in order to describe the infinite twist of the local transition map $\Psi^l$ near $C^s_{\tau_1}$.

Take $\delta>0$ small. Consider the real-analytic open annuli $R^s_\delta, R^u_\delta \subset  K^{-1}(E)$ given in Moser's coordinates $(q_1,q_2,p_1,p_2)$ by
\begin{equation}\label{R}
\begin{aligned}
R^s_\delta&:=\left\{(q_1,q_2,p_1,p_2):q_1=\delta,-\frac{\delta}{2}< p_1< \frac{\delta}{2}\right\}\cap K^{-1}(E),\\
R^u_\delta&:=\left\{(q_1,q_2,p_1,p_2):p_1=\delta,-\frac{\delta}{2}< q_1< \frac{\delta}{2}\right\} \cap K^{-1}(E),
\end{aligned}
\end{equation}
and let $A^s_\delta \subset R^s_\delta$ and $A^u_\delta \subset R^u_\delta$ be defined by
\begin{equation}\label{secoes}
\begin{aligned}
A^s_\delta&:=\left\{(q_1,q_2,p_1,p_2):q_1=\delta,0< p_1< \frac{\delta}{2}\right\} \cap K^{-1}(E),\\
A^u_\delta&:=\left\{(q_1,q_2,p_1,p_2):p_1=\delta,0< q_1< \frac{\delta}{2}\right\} \cap K^{-1}(E).
\end{aligned}
\end{equation}
See Figure \ref{fig_secoes}. Both $R_\delta^s$ and $R_\delta^u$ are transverse to the Hamiltonian vector field $X_K$.
Observe that $A^s_\delta$ and $A^u_\delta$ correspond to non-transit trajectories of the Hamiltonian flow, i.e., those which do not cross
the separating 2-sphere $\{q_1+p_1=0\} \cap K^{-1}(E)$. In fact, $A^s_\delta$ is mapped onto $A^u_\delta$ by the local Hamiltonian flow.

Since $\partial_{I_1}\bar K(0,0)=-\alpha\neq 0$, it follows from the implicit function theorem that there exists a real-analytic function $f$ defined in a neighborhood of $(0,0)\in \R^2$ such that
\begin{equation}\label{I1}
I_1=f(I_2,E)=\frac{\omega}{\alpha}I_2 -\frac{E}{\alpha} +O(I_2^2+E^2)  \mbox{ with } \bar K(f(I_2,E),I_2)=E.
\end{equation}
For the sake of simplicity we  omit the dependence of $f$ on $E$ since $E>0$ small is fixed.

Let $0<I_2^\sharp<I_2^c<I_2^*$ be so that $$f(I_2^\sharp)=-\frac{\delta^2}{2}, f(I_2^c)=0 \mbox{ e } f(I_2^*)=\frac{\delta^2}{2}.$$
Taking $\delta>0$ sufficiently small, we see that the local stable manifold $W^s_{\rm loc}(P_{2,E})$ intersects the annulus $R_\delta^s$ along the real-analytic circle
\begin{equation}\label{Ss}
S^s_\delta:=\{q_1=\delta,p_1=0, q_2^2+p_2^2=2I_2^c\}
\end{equation}
and the local unstable manifold $W^u_{\rm loc}(P_{2,E})$ intersects the annulus $R_\delta^u$ on the real-analytic circle
\begin{equation}\label{Su}
S^u_\delta:=\{q_1=0,p_1=\delta, q_2^2+p_2^2=2I_2^c\}.
\end{equation}
See Figure \ref{fig_secoes}. %Note that $S^s_\delta$ and $S^u_\delta$ correspond to asymptotic orbits to $P_{2,E}$ as time goes to $+\infty$ and $-\infty$, respectively.

\begin{figure}[ht!]
  \centering
  \includegraphics[width=0.4\textwidth]{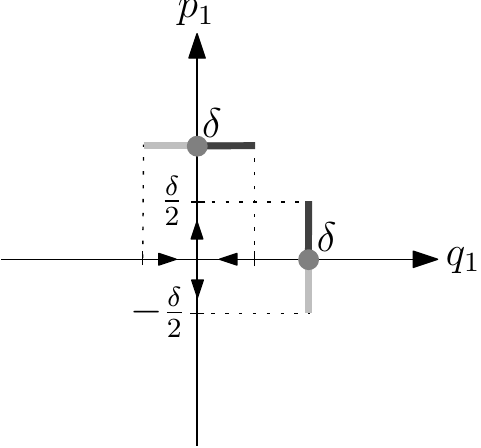}
  \caption{The transverse sections $R^s_\delta$ and $R_\delta^u$, defined in \eqref{R}, projected to the plane $(q_1,p_1)$. Each dark gray line represents the projection of one of the annuli $A^s_\delta$ and $A_\delta^u$, defined in \eqref{secoes}, and each gray dot represents the projection of one of the circles $S^s_\delta$ and $S_\delta^u$, given in \eqref{Ss} and \eqref{Su}, respectively.}
  \label{fig_secoes}
\hfill
\end{figure}

From now on it will be convenient to consider angle-action coordinates $(\theta,I_2)$ instead of the coordinates $(q_2,p_2)$, which are defined by the following relations:
$$q_2=\sqrt{2I_2} \cos\theta \mbox{ and } p_2=\sqrt{2I_2} \sin\theta.$$
In such coordinates, the circles $S^s_\delta$ and $S^u_\delta$ correspond to
\begin{equation}\label{Sc}
S_{c}:=\left\{(\theta,I_2) \in (\R/2\pi \Z)\times \R^+:\, I_2=I_2^c\right\},
\end{equation}
the open annuli $A^s_\delta$ and $A_\delta^u$ correspond to
\begin{equation}\label{Ac*}
A_{c}^{*}:=\left\{(\theta,I_2) \in (\R/2\pi \Z)\times \R^+:\, I_2^c< I_2<I_2^*\right\}
\end{equation}
and the open annuli $R^s_\delta$ and $R_\delta^u$ correspond to
\begin{equation}\label{Asust*}
A_{\sharp}^{*}:=\left\{(\theta,I_2) \in (\R/2\pi \Z)\times \R^+:\, I_2^\sharp < I_2 <I_2^*\right\}.
\end{equation}

We introduce the following real-analytic symplectic charts
\begin{equation}\label{eq_fisudelta0}\varphi_\delta^s: A_\sharp^*\to R^s_\delta \mbox{ and } \varphi_\delta^u : A_\sharp^*\to R^u_\delta\end{equation}
given by
\begin{equation}\label{eq_fisudelta}
\begin{aligned}
\varphi_\delta^s(\theta,I_2)=&\left(\delta, \sqrt{2I_2} \cos\theta,\frac{f(I_2)}{\delta}, \sqrt{2I_2} \sin\theta\right),\\
\varphi_\delta^u(\theta,I_2)=&\left(\frac{f(I_2)}{\delta}, \sqrt{2I_2} \cos\theta,\delta, \sqrt{2I_2} \sin\theta\right),
\end{aligned}
\end{equation}
which satisfy $\varphi^s_\delta(S_c)=S^s_\delta$, $\varphi^s_\delta(A_c^*)=A^s_\delta$, $\varphi^u_\delta(S_c)=S^u_\delta$ and $\varphi^u_\delta(A_c^*)=A^u_\delta$.
As a result, we obtain the real-analytic local model $(A_\sharp^*,A_c^*,S_c)$ for both $(R^s_\delta, A^s_\delta, S^s_\delta)$ and $(R^u_\delta, A^u_\delta, S^u_\delta)$. Notice that $(\varphi_\delta^u)^*\omega_0= (\varphi_\delta^s)^* \omega_0= d\theta \wedge dI_2$, where $\omega_0 = dp_1 \wedge dq_1 + dp_2 \wedge dq_2$.

Since $R^u_\delta, R^s_\delta, D_{\tau_0}$ and $D_{\tau_1}$ are transverse to the Hamiltonian vector field, one may use the positive Hamiltonian flow to obtain symplectomorphisms
\begin{equation}\label{eq_deltasu}
\psi^s_{\tau_1,\delta}: V_{\tau_1}^s \to R_\delta^s\mbox{ and }\psi^u_{\delta,\tau_0}: R_\delta^u \to V_{\tau_0}^u
\end{equation}
for $\delta>0$ sufficiently small, where $V^s_{\tau_1} \subset D_{\tau_1}$ and $V^u_{\tau_0} \subset D_{\tau_0}$ are suitable annular neighborhoods of the circles $C_{\tau_1}^s\subset D_{\tau_1}$ and $C_{\tau_0}^u\subset D_{\tau_0}$, respectively, such that $\psi^s_{\tau_1,\delta}(C_{\tau_1}^s)=S_\delta^s$ and $\psi^u_{\delta,\tau_0}(S_\delta^u)=C_{\tau_0}^u$.
Denoting
\begin{equation}\label{N}
N_{\tau_1}^s:=V_{\tau_1}^s \setminus B_{\tau_1}^s \mbox{ and } N_{\tau_0}^u:=V_{\tau_0}^u \setminus B_{\tau_0}^u,
\end{equation}
we observe that $\psi^s_{\tau_1,\delta}(N_{\tau_1}^s)=A_\delta^s$ and $\psi^u_{\delta,\tau_0}(A_\delta^u)=N_{\tau_0}^u$.

Using the maps defined in \eqref{eq_fisudelta0} and \eqref{eq_deltasu}, we can consider from now on the angle-action coordinates $(\theta,I_2) \in A_\sharp^*$ on both annular neighborhoods $V_{\tau_1}^s\subset D_{\tau_1}$ and $V_{\tau_0}^u\subset D_{\tau_0}$, in such a way that $\psi_{\tau_1,\delta}^s$ and $\psi_{\delta,\tau_0}^u$ correspond to the identity maps in coordinates $(\theta,I_2)$.
All this construction %, summarized in Figure \ref{fig_notacoes},
provides a local model $(A_\sharp^*,A_c^*,S_c)$ for both $(V^s_{\tau_1}, N^s_{\tau_1},$ $C^s_{\tau_1})$ and $(V^u_{\tau_0}, N^u_{\tau_0}, C^u_{\tau_0})$. The importance of this local model is that the transition maps induced by the Hamiltonian flow are real-analytic.

%\begin{figure}[ht!]
 % \centering
 % \includegraphics[width=0.85\textwidth]{notacoes.pdf}
 % \caption{Real-analytic local models for neighborhoods of $C^s_{\tau_1}\subset D_{\tau_1}$ and $C^u_{\tau_0}\subset D_{\tau_0}$.}
 % \label{fig_notacoes}
%\hfill
%\end{figure}

Now we give a description of the local transition map $\Psi^l: D_{\tau_1} \setminus B^s_{\tau_1} \to D_{\tau_0}\setminus B^u_{\tau_0}$, defined in \eqref{psil}, near the inner boundary circles $C^s_{\tau_1} = \partial B^s_{\tau_1}$ and $C^u_{\tau_0} = \partial B^u_{\tau_0}$,  making use of the angle-action coordinates $(\theta,I_2)\in A_c^*$ on both $N_{\tau_1}^s\subset D_{\tau_1}$ and $N_{\tau_0}^u\subset D_{\tau_0}$.

By observing the behavior of the flow in local coordinates $(q_1,q_2,p_1,p_2)$, see Figure \ref{fig_selacentro}, we notice that positive trajectories starting at $A^s_\delta$ must hit $A_\delta^u$. Since these solutions satisfy \eqref{eqK}, one can compute the time that an orbit starting at $A_\delta^s$ takes to reach $A_\delta^u$. This is given in angle-action coordinates by
$$T(\theta,I_2)=t(I_2):=-\frac{1}{\bar\alpha(I_2)} \ln \frac{f(I_2)}{\delta^2},$$
where \begin{equation}\label{alphabar}\bar\alpha(I_2)=\alpha-\partial_{I_1}R(f(I_2),I_2).\end{equation}
Note that $f(I_2)$ is positive on $A_c^*$ and,  as $I_2 \to (I_2^c)^+$, $\bar\alpha(I_2)$ converges to $\alpha-\partial_{I_1}R(0,I_2^c)>0$ and $t(I_2)$ converges to $+\infty$.

The restriction $\Psi^l: N_{\tau_1}^s \to N_{\tau_0}^u$ of the local transition map is represented by a map $l:A_\delta^s \to A_\delta^u$, which in coordinates $(\theta,I_2)\in A_c^*$ is written as
\begin{equation}\label{local}
\begin{aligned}
l &: A_c^* \to A_c^*\\
l(\theta,I_2)&= (\theta+\Delta \theta(I_2), I_2),
\end{aligned}
\end{equation}
where the variation $\Delta \theta$ is given by
\begin{equation}\label{variacao}
\Delta \theta(I_2)=-\bar\omega(I_2)\,t(I_2)=\frac{\bar\omega(I_2)}{\bar\alpha(I_2)} \ln \frac{f(I_2)}{\delta^2}.
\end{equation}
Here \begin{equation}\label{omegabar}\bar\omega(I_2)=\omega+\partial_{I_2}R(f(I_2),I_2),\end{equation} which converges to $\omega+\partial_{I_2}R(0,I_2^c)$ $>0$ as $I_2\to (I_2^c)^+$. Notice that $l$ preserves $I_2$.

Since $f$ is real-analytic and has a simple zero  at $I_2^c>0$, one can find a real-analytic function $\eta$ satisfying
\begin{equation}\label{eta}
f(I_2)=(I_2-I_2^c)\eta(I_2) \mbox{ with } \eta(I_2^c)\neq 0.
\end{equation}
From \eqref{variacao} and \eqref{eta} we obtain the following estimate for $\Delta\theta$
\begin{equation}\label{variacaolambda}
\Delta \theta(I_2)=\frac{\bar \omega(I_2)}{\bar \alpha(I_2)} \ln (I_2-I_2^c)+\Lambda(I_2),
\end{equation}
where
\begin{equation}\label{Lambda}\Lambda(I_2)=\frac{\bar\omega(I_2)}{\bar\alpha(I_2)}\ln \frac{\eta(I_2)}{\delta^2}\end{equation} is a real-analytic function near $I_2^c$. Since $f(I_2)>0$ on $A_c^*$, we have that $\eta(I_2)>0$ on this domain.

Now we use  \eqref{variacaolambda} in order to prove that the local transition map $l$ has a monotone infinite twist near the inner boundary component $S_c$ of the annulus $A_c^*$.

\begin{lem}\label{twist}
Let $\gamma:[0,1) \to A_c^* \cup S_c$ be a real-analytic curve such that $\gamma(0) \in S_c$ and $\gamma(t) \in A_c^*$ for all $t \in (0,1)$. Then, for all $\epsilon>0$ sufficiently small, the curve $\{l\circ\gamma(t): t\in (0,\epsilon)\}$ is a real-analytic spiral turning monotonically around $S_c$ in the clockwise direction and accumulating on $S_c$  as $t\to 0^+$. See Figure \ref{fig_espiral}.
\end{lem}

\begin{figure}[h!!]
  \centering
  \includegraphics[width=0.35\textwidth]{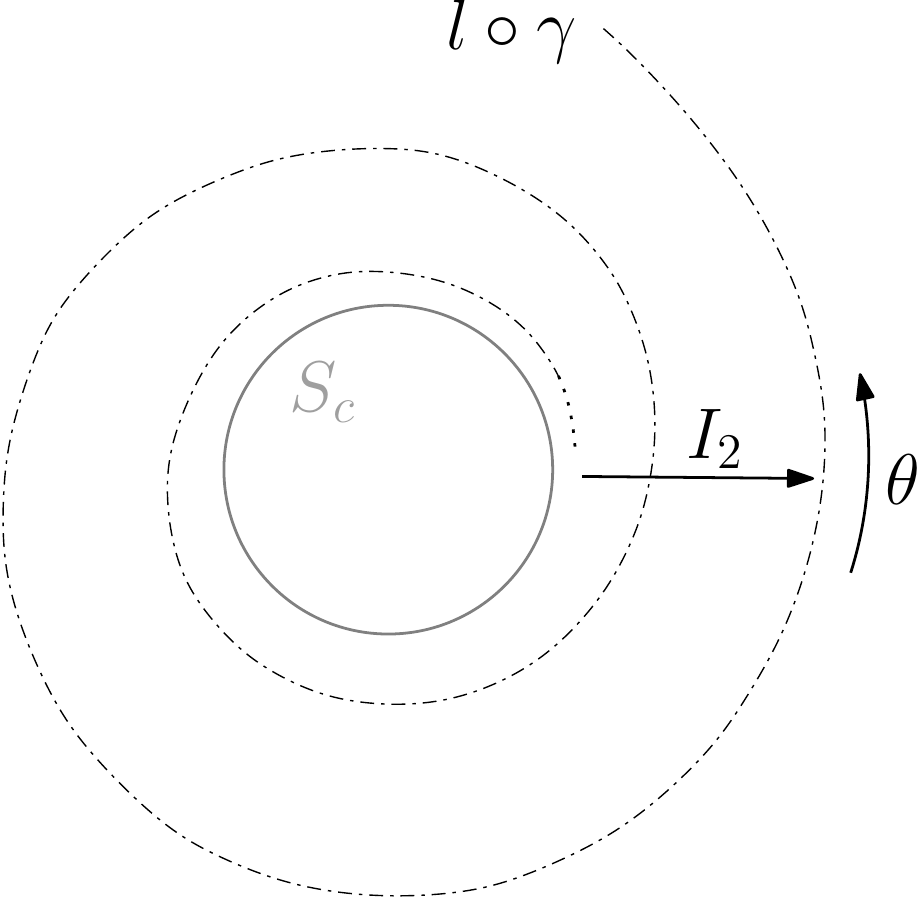}
  \caption{Infinite twist of the local transition map $l$ near $S_c$.}
  \label{fig_espiral}
\hfill
\end{figure}

\begin{proof}
The point $\gamma(0)\in S_c$ is represented by $(\theta_0, I_2^c)$ and the curve $\gamma$ can be written as $\gamma(t)=(\theta(t),I_2(t))$, where $\theta$ and
%:[0,1)\to \R$ and
$I_2$
%:[0,1)\to \left[I_2^c,I_2^*\right)$
are real-analytic functions satisfying $\theta(0)=\theta_0$ and $I_2(0)=I_2^c$. Since $\gamma$ is real-analytic, $\gamma(0) \in S_c$ and $\gamma(t) \in A_c^*$ for all $t \in (0,1)$, we can find $\epsilon_0>0$ such that  the restriction of $I_2$ to the interval $[0,\epsilon_0)$ is a strictly increasing function. In fact, since $I_2(t)$ is real-analytic, we have $$I_2(t)-I_2^c = c_0t^n +O(t^{n+1})$$ for some $c_0>0$, $n\in \N^*$ and all $t\geq 0$ small. Hence $$I_2'(t) =nc_0t^{n-1} +O(t^n).$$

 On the other hand, \eqref{variacaolambda} asserts that the variation $\Delta \theta$ on $\{\gamma(t), t\in (0,\epsilon_0)\}$ is given by
$$\Delta \theta(I_2(t))=\frac{\bar\omega(I_2(t))}{\bar\alpha(I_2(t))} \ln (I_2(t)-I_2^c)+\Lambda(I_2(t)),$$
where $\Lambda(I_2(t))$ is a real-analytic function near $t=0$. Hence
$$\begin{aligned}
\frac{d}{dt}[\theta(t)+\Delta \theta(I_2(t))] & = \theta'(t)+\frac{d}{dt}\left[\frac{\bar\omega(I_2(t))}{\bar\alpha(I_2(t))}\right] \ln (I_2(t)-I_2^c)\\
& +\frac{\bar\omega(I_2(t))}{\bar\alpha(I_2(t))} \frac{I_2'(t)}{I_2(t)-I_2^c}+ \Lambda'(I_2(t))I_2'(t).
\end{aligned}$$
Since $\theta$ and $\Lambda$ are both real-analytic near $t=0$ and since $\frac{\bar \omega(I_2^c)}{\bar \alpha(I_2^c)}>0$, we can find $0<\epsilon_1<\epsilon_0$  and a constant $c>0$ so that
$$
\frac{d}{dt}[\theta(t)+\Delta \theta(I_2(t))] > O(\ln t) + \frac{2c}{n}\frac{nc_0t^{n-1} +O(t^n)}{c_0t^n+O(t^{n+1})} > c\frac{1}{t}>0
$$
for all $t \in (0,\epsilon_1)$.

We conclude that $\theta(t)+ \Delta \theta(I_2(t))$ decreases monotonically to $-\infty$ as $t \to 0^+$ and, therefore, the curve
$$l\circ \gamma(t)=(\theta(t)+ \Delta \theta(I_2(t)), I_2(t)), t \in (0,\epsilon_1),$$ is a spiral that turns monotonically in the clockwise direction around $S_c$ and accumulates on $S_c$ as $t \to 0^+$.
\end{proof}

\section{Positive topological entropy}\label{sec_naocoincide}

We keep using the notations established in Sections \ref{sec_existenciahomoc} and \ref{sec_twist}.
Let $\Psi^g$ and $\Psi^l$ be the global and local symplectomorphisms defined in \eqref{psig} and \eqref{psil}, respectively. According to Proposition \ref{prop_homoclinic2}, we can find $N\in\N$ such that $C_N^u \cap C^s_{\tau_1}\neq \emptyset$, where $C_N^u=(\Psi^g \circ \Psi^l)^{N-1} \circ \Psi^g(C^u_{\tau_0})$, $C^u_{\tau_0}$ and $C^s_{\tau_1}$ are the intersections of $W^u_{\rm loc}(P_{2,E})$ and $W^s_{\rm loc}(P_{2,E})$ with $D_{\tau_0}$ and $D_{\tau_1}$, respectively, and $0<\tau_0<\tau_1<1$, with $\tau_0$ near $0$ and $\tau_1$ near $1$.

%Recalling the notations, let $C^u_{\tau_0}$ be the circle given by the intersection between $D_{\tau_0}$ and the branch $W^u_{\rm loc}(P_{2,E})$ of the local unstable manifold of $P_{2,E}$ inside $\dot S_E$, for $0<\tau_0<1$ sufficiently close to $0$, and let $C^s_{\tau_1}$ be the circle of intersection between $D_{\tau_1}$ and the branch $W^s_{\rm loc}(P_{2,E})$ of the local stable manifold of $P_{2,E}$ inside $\dot S_E$, for $0<\tau_1<1$ sufficiently close to $1$. We denote by $B^u_{\tau_0}\subset D_{\tau_0}$ and $B^s_{\tau_1}\subset D_{\tau_1}$ the closed disks bounded by $C^u_{\tau_0}$ and $C^s_{\tau_1}$ respectively. Observe that $B^u_{\tau_0} \setminus C^u_{\tau_0}$ and $B^s_{\tau_1} \setminus C^s_{\tau_1}$ correspond to transit trajectories while $C^u_{\tau_0}$ and $C^s_{\tau_1}$ correspond to asymptotic orbits to $P_{2,E}$.
%Let $\Psi^g$ and $\Psi^l$ be the global and local symplectomorphisms defined in \eqref{psig} and \eqref{psil}. From Lemma \ref{prop_homoclinic2}, we know there exists $N\in \N$ so that
%$C^u_N \cap C^s_{\tau_1}\neq \emptyset,$ where $C^u_N=(\Psi^g \circ \Psi^l)^{N-1} \circ \Psi^g(C^u_{\tau_0})$.

In this section we assume that the branches of the invariant manifolds $W^u(P_{2,E})$ and $W^s(P_{2,E})$ inside $S_E$ do not coincide. In particular, this implies that $C^u_N$ and $C^s_{\tau_1}$ do not coincide as well. In this case, we shall prove that even if $C^u_N$ and $C^s_{\tau_1}$ do not intersect transversely, we may consider higher iterates of the map $\Psi^g \circ \Psi^l$ in order to find infinitely many transverse intersections of $W^u(P_{2,E})$ and $W^s(P_{2,E})$ in $S_E$. As a consequence of these transverse intersections, we obtain infinitely many periodic orbits and infinitely many homoclinics to $P_{2,E}$ inside $S_E$. Moreover, the flow on  $H^{-1}(E)$ has positive topological entropy. Such dynamical properties follow from the existence of an invariant subset $\Lambda_E\subset S_E$ so that the flow restricted to $\Lambda_E$ semi-conjugates to a Bernoulli shift of infinite type, see \cite{BGS} and \cite{moser2}.

%A point at which the stable and unstable manifolds of $P_{2,E}$ intersect transversely (or tangentially) is called a transverse (or tangential) homoclinic point to $P_{2,E}$.

As we have seen in Section \ref{sec_twist}, $(A_\sharp^*, A_c^*, S_c)$ is a local model for $(V^s_{\tau_1}, N^s_{\tau_1}, C^s_{\tau_1})$ and $(V^u_{\tau_0}, N^u_{\tau_0}, C^u_{\tau_0})$, where $A_\sharp^*$ is the annular neighborhood of $S_c$ defined in \eqref{Asust*}, with the circle $S_c$ defined in \eqref{Sc} being the inner boundary circle of the annulus $A_c^*$ given in  \eqref{Ac*}, and $V^s_{\tau_1} \subset D_{\tau_1}$, $V^u_{\tau_0}\subset D_{\tau_0}$ are, respectively, neighborhoods of the circles $C^s_{\tau_1}$ and $C^u_{\tau_0}$, which  are the inner boundary circles of the annuli $N_{\tau_1}^s$ and $N_{\tau_0}^u$ given in \eqref{N}, respectively.

Since the Hamiltonian function $H$ and Moser's coordinates are real-analytic, $C^u_N \cap V^s_{\tau_1}$ is represented by a real-analytic curve $\sigma_N^u \subset A_\sharp^*$ (possibly with many connected components) and  homoclinic points in $C^u_N \cap C^s_{\tau_1}$ correspond to intersection points in $\sigma_N^u \cap S_c$.

We are assuming that $C^u_N$ and $C^s_{\tau_1}$ intersect each other but do not coincide. It follows that each point in $\sigma_N^u \cap S_c$ is isolated since $\sigma_N^u$ and $S_c$ are real-analytic curves. Moreover, since the disks $B_N^u$ and $B^s_{\tau_1}$ have the same symplectic area $T_{2,E}>0$, we conclude that $C_N^u \cap N_{\tau_1}^s \neq \emptyset$ and, therefore, $\sigma_N^u \cap A_c^* \neq \emptyset$.

\begin{definition}\label{ext}
We say that $p\in\sigma_N^u \cap S_c$ is a special homoclinic point if one of the following properties holds:
 \begin{itemize}
 \item[(i)]  the connected component $\gamma_p$ of $\sigma_N^u$ containing  $p$ crosses the circle $S_c$ at $p$. This means that there are points in $\gamma_p$ arbitrarily close to $p$ satisfying $I_2 < I_2^c$ and points arbitrarily close to $p$ satisfying $I_2 > I_2^c$.
%(or, equivalently, the interior of $B_N^u$ intersects $C^s_{\tau_1}$)
%contains small arcs $\gamma_u^+, \gamma_u-\subset A_c^* \cup S_c$ with a common end point $p$ satisfying $\gamma_u^+ \setminus \{p\}\subset A_c^*$ and $\gamma_u^- \setminus \{p\} \subset A_\#^* \setminus (A_c^* \cup S_c)$.
\item[(ii)] $\sigma_N^u$ does not contain points satisfying $I_2 < I_2^c$.
% \item[(ii)] is such that $\gamma_p \setminus \{p\}$ lies in $A_c^*$ and, near $p$,  $\Delta_N^u$  lies in $A_c^*$.
 %is any (necessarily tangential) intersection point of $\sigma_N^u$ with $S_c$. In this case the interior of $B_N^u$ does not intersect $C^s_{\tau_1}$.
 \end{itemize}
\end{definition}
\begin{figure}[ht!]
  \centering
  \includegraphics[width=0.75\textwidth]{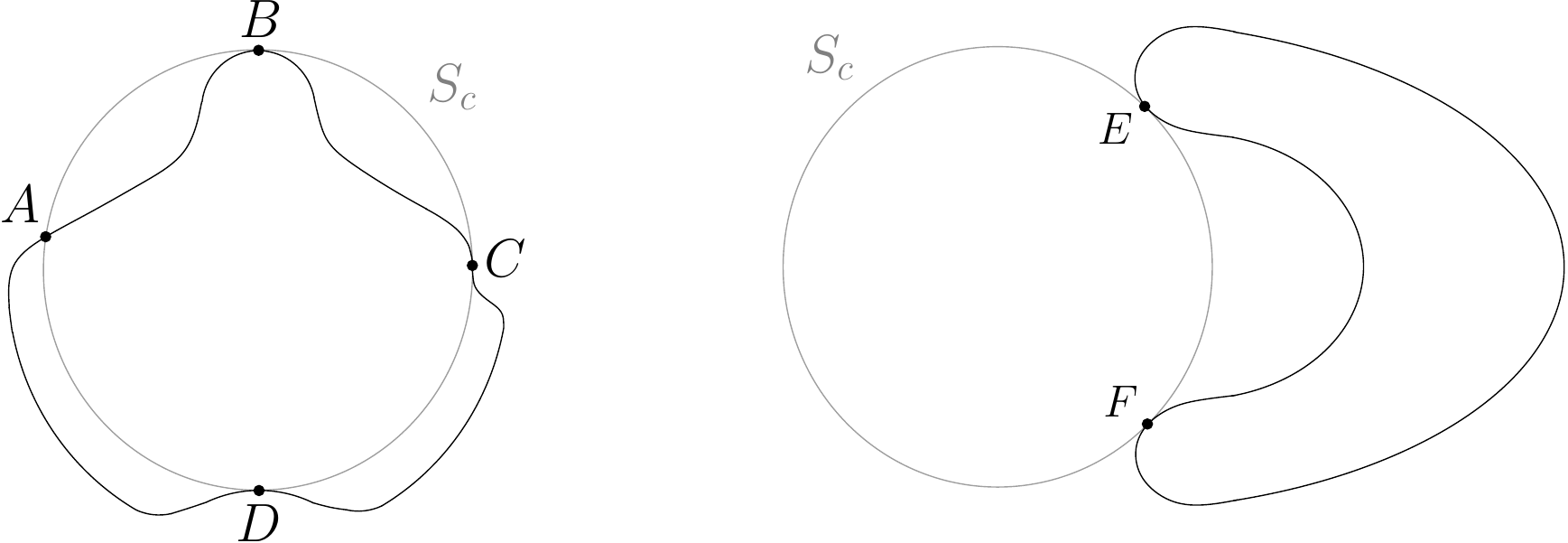}
  \caption{In this figure $A$, $C$, $E$ and $F$ are special homoclinic points: $A$ and $C$ are crossing points as in Definition \ref{ext} {\it (i)}, and $E$ and $F$ are special as in Definition \ref{ext} {\it (ii)}. The points $B$ and $D$ are not special homoclinic points.}
  \label{fig_special}
\hfill
\end{figure}
We can always find a special homoclinic point $p\in \sigma_N^u \cap S_c$. In fact, if $(B_N^u \setminus C^u_N) \cap (B^s_{\tau_1} \setminus C^s_{\tau_1}) \neq \emptyset$ then a crossing point as in (i) must exist. This follows from the fact that $B^u_N$ and $B^s_{\tau_1}$ do not coincide and have the same symplectic area $T_{2,E}>0$. Otherwise, if $(B_N^u \setminus C^u_N) \cap (B^s_{\tau_1} \setminus C^s_{\tau_1})=\emptyset$ then a special point as in (ii) must exist. See Figure \ref{fig_special}.

Observe that since $\sigma_N^u$ and $S_c$ are real-analytic curves, either $\sigma_N^u$ and $S_c$ cross each other transversely at $p$
or $\sigma_N^u$ meets $S_c$ tangentially at $p$ with contact of finite order.

The dynamics near a transverse homoclinic point is very rich and rather complicated. In particular, it forces the existence of infinitely many homoclinics  and infinitely many periodic orbits. This is well known since Poincar\'e \cite{poincare}, see also Smale \cite{smale}.

In the following we prove that any neighborhood of a special homoclinic point (in the sense of Definition \ref{ext}) contains infinitely many transverse homoclinic points, as outlined by Conley\footnote{Churchill and Rod also addressed this question in a more general setup, see \cite[Theorem 1.1, Remark 1.8 (b) and (c)]{CR}.} in \cite{conley}.
More specifically, we use the real-analyticity of the flow and the infinite twist of the local transition map in order to show that a special homoclinic point in $\sigma_N^u \cap S_c$ implies the existence of infinitely many  transverse intersections between $(g \circ l)^N(\sigma_N^u \cap A_c^*)$ and $S_c$, where $g$ and $l$ are local representations of the global and the local transition maps $\Psi^g$ and $\Psi^l$, respectively, in angle-action coordinates $(\theta,I_2)$. For a more precise expression of $g$ in coordinates $(x,y)=(\theta, I_2 -I_2^c)$, see the proof of Lemma \ref{lem_global} below.

%In the case studied in the present section, where $C^u_N$ and $C^s_{\tau_1}$ intersect without coinciding, we can always find a external intersection between the branches of $W^s_E(P_{2,E})$ and $W^u_E(P_{2,E})$ contained in $\dot S_E$.

Consider a special homoclinic point $p \in \sigma_N^u \cap S_c$.
%The real-analyticity of these curves ensures that this intersection has contact of finite order.
Let us denote by $\gamma_u \subset A_c^* \cup S_c$ a real-analytic arc contained in $\sigma_N^u$ which has one end point in $p \in S_c$ and such that $\dot\gamma_u:=\gamma_u \setminus \{p\}$ is contained in $A_c^*$. Observe that
$$\gamma_u \subset F^{N-1} \circ g(S_c),$$ where $F:=g \circ l$.%, i.e., $\gamma_u$ corresponds to a curve which is carried backward by the Hamiltonian flow to $S_c$.

Now let $\gamma_s \subset  A_c^* \cup S_c$ be a real-analytic arc with an end point $q \in S_c$ such that $\gamma_s \setminus \{q\}$ is contained in $A_c^*$. Assume that $\gamma_s$ satisfies %is carried forward by the Hamiltonian flow to $S_c$, more precisely,
$$F^{N-1} \circ g(\gamma_s)\subset S_c \mbox{ and } F^{N-1} \circ g(q)=p.$$ The hypothesis of $p$ being a special homoclinic point implies that the arcs $\gamma_u$ and $\gamma_s$ exist. %See Figure \ref{fig_curvasgama}. Observe that $l(\dot \gamma_u) \cap \gamma_s$ corresponds to new homoclinic points.
%\begin{figure}[ht!]
% \centering
%\includegraphics[width=0.8\textwidth]{curvasgama2.pdf}
%\caption{The real-analytic curves $\gamma_u, \gamma_s \subset A_c^* \cup S_c$.}
%\label{fig_curvasgama}
%\hfill
%\end{figure}

Using Lemma \ref{twist} we see that $l(\dot\gamma_u)$ is a real-analytic spiral turning monotonically around $S_c$ in the clockwise direction  and accumulating on $S_c$. Therefore, $l(\dot\gamma_u)$ intersects $\gamma_s$ infinitely many times. Note that each of these intersections corresponds to a homoclinic orbit to $P_{2,E}$. Now we prove that such intersections are transverse near $S_c$.

The points $p\in \gamma_u \cap S_c$ and $q\in \gamma_s \cap S_c$ are represented in angle-action coordinates by $(\vartheta_p, I_2^c)$ and $(\vartheta_q, I_2^c)$. We can assume that $\vartheta_p=\vartheta_q=0$ without loss of generality.
Since $p$ is an isolated intersection point, we can write
$$\gamma_u=\{(\vartheta_{\gamma_u}(I_2),I_2): I_2 \in [I_2^c,I_2^c+\epsilon)\}$$
for some $\epsilon>0$ small, where $\vartheta_{\gamma_u}$ is a real-analytic function in $(I_2^c,I_2^c+\epsilon)$ satisfying $\vartheta_{\gamma_u}(I_2^c)=0$. As mentioned before either $\sigma_N^u$ and $S_c$ cross each other transversely at $p$ or they meet tangentially at $p$ with contact of finite order.
Then either $\vartheta_{\gamma_u}\equiv 0$ or $\vartheta_{\gamma_u}$ can be written as
%, which is continuous in $[I_2^c,I_2^c+\epsilon)$ and real-analytic in $(I_2^c,I_2^c+\epsilon)$,
\begin{equation}\label{vartetagamau}
\vartheta_{\gamma_u}(I_2)=(I_2-I_2^c)^n \beta(I_2),
\end{equation}
where $\beta$ is a real-analytic function satisfying $\beta(I_2^c) \neq 0$, $a\geq 0$ and $n>0$ is either an integer in case the intersection is transverse or $n=\frac{1}{k}$ for an integer $k>0$ in case the intersection is tangential with contact of finite order.

As we have seen in Section \ref{sec_twist}, the local transition map $l: A_c^* \to A_c^*$ is defined by $l(\theta,I_2)=(\theta+\Delta\theta(I_2),I_2)$, where the variation $\Delta\theta(I_2)$ is given in \eqref{variacaolambda}. Thus, the real-analytic curve $l(\dot\gamma_u)$ assumes the following form
$$l(\dot\gamma_u)=\{(\vartheta_{l(\dot \gamma_u)}(I_2):=\vartheta_{\gamma_u}(I_2)+\Delta\theta(I_2), I_2): I_2 \in (I_2^c,I_2^c+\epsilon)\}.$$
From \eqref{variacaolambda} and \eqref{vartetagamau} we obtain, for $I_2-I_2^c>0$ small,
$$\vartheta_{l(\dot \gamma_u)}(I_2)=(I_2-I_2^c)^n \beta(I_2)+\frac{\bar \omega(I_2)}{\bar \alpha(I_2)} \ln (I_2-I_2^c)+\Lambda(I_2),$$
where $\Lambda(I_2)$ is a real-analytic function near $I_2^c$. It follows that
$$\begin{aligned} \frac{d\vartheta_{l(\dot \gamma_u)}}{dI_2}(I_2) & =n(I_2-I_2^c)^{n-1}\beta(I_2)+(I_2-I_2^c)^n\beta'(I_2)\\ & +\frac{\bar \omega(I_2)}{\bar \alpha(I_2)} \frac{1}{I_2-I_2^c}+\partial_{I_2}\!\left(\frac{\bar \omega(I_2)}{\bar \alpha(I_2)}\right)\ln (I_2 - I_2^c)+\Lambda'(I_2).\end{aligned}$$
Hence, in both cases $n \in \N^*$ or $n=\frac{1}{k}$ with $k \in \N^*$, we  find a  constant $c_1>0$ so that \begin{equation}\label{c1}\frac{d\vartheta_{l(\gamma_u)}}{dI_2}(I_2)>c_1 \frac{1}{I_2-I_2^c}\end{equation} for all $I_2- I_2^c>0$ sufficiently small. Inequality \eqref{c1} also holds in the case $\vartheta_{\gamma_u} \equiv 0$.

Now let $\vartheta_{\gamma_s}(I_2)$ be a function (real-analytic in $(I_2^c,I_2^c+\epsilon)$) determined by $\gamma_s$ as follows
$$\gamma_s=\{(\vartheta_{\gamma_s}(I_2), I_2): I_2 \in [I_2^c,I_2^c+\epsilon)\},$$
where $\vartheta_{\gamma_s}(I_2^c)=0$. As before, either $\vartheta_{\gamma_s}\equiv 0$ or $\vartheta_{\gamma_s}$ can be written as
\begin{equation}\label{vartetagamas}
\vartheta_{\gamma_s}(I_2)=(I_2-I_2^c)^m \varsigma(I_2),
\end{equation}
where $\varsigma$ is a real-analytic function satisfying $\varsigma(I_2^c) \neq 0$ and $m>0$ is either an integer or assumes the form $m=\frac{1}{j}$ for some integer $j>1$, depending on the way $\gamma_s$ meets $S_c$ at the point $q$. In both cases we can choose  $\lambda \in (0,1)$ so that
\begin{equation}\label{eq2}\frac{d\vartheta_{\gamma_s}}{dI_2}(I_2)=m(I_2-I_2^c)^{m-1}\varsigma(I_2)+(I_2-I_2^c)^m \varsigma'(I_2)<c_1\frac{1}{(I_2-I_2^c)^{1-\lambda}},\end{equation}
for all $I_2-I_2^c>0$ small. Inequality above also holds in the case $\vartheta_{\gamma_s}\equiv 0$.

Using \eqref{c1} and \eqref{eq2} we conclude $$\frac{d\vartheta_{l(\dot \gamma_u)}}{dI_2}(I_2)>\frac{d\vartheta_{\gamma_s}}{dI_2}(I_2)$$
for all $I_2>I_2^c$ near $I_2^c$, proving that $l(\dot\gamma_u)$ intersects $\gamma_s$ transversely at infinitely many points approaching $(\theta, I_2)=(0, I_2^c)$. See Figure \ref{fig_inter_transversal}.

\begin{figure}[ht!]
  \centering
  \includegraphics[width=0.5\textwidth]{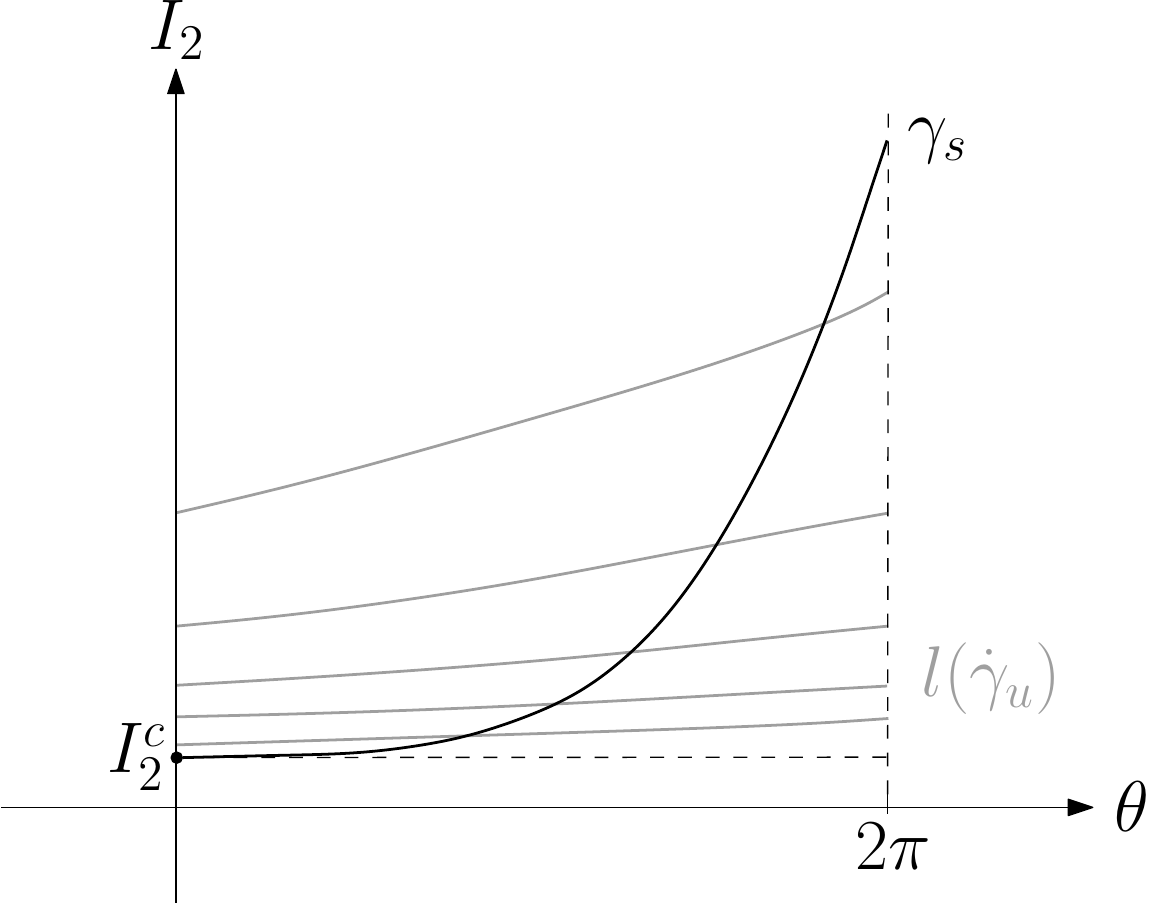}
  \caption{Representation of the curves $l(\dot\gamma_u)$ and $\gamma_s$ in angle-action coordinates $(\theta, I_2)$. }
  \label{fig_inter_transversal}
\hfill
\end{figure}

Transverse intersection points between $l(\dot\gamma_u)$ and $\gamma_s$ correspond to transverse homoclinic orbits to $P_{2,E}$. Therefore we have proved that $P_{2,E}$ admits infinitely many transverse homoclinic orbits in $\dot S_E$. It is well known that the existence of a transverse homoclinic orbit implies infinitely many nearby periodic orbits and positivity of topological entropy. See \cite[Chapter III]{moser2} for a discussion of this result and other interesting dynamical properties implied by the transverse homoclinic orbits such as the existence of an invariant subset of $\dot S_E$ so that the flow restricted to it is semi-conjugated to a Bernoulli shift with infinitely many symbols.

%Moreover, in this case we can use the infinite twist of the local transition map in order to obtain a semi-conjugacy between the Hamiltonian flow, restricted to an invariant subset $\Lambda_E\subset S_E$, and a Bernoulli shift of infinite type. We refer to \cite{BGS} for a detailed construction of such a semi-conjugacy, which is based on. Hence the Hamiltonian flow on $S_E$ has positive topological entropy. Many dynamical information can be obtained from this chaotic subsystem. For instance, there exist infinitely many periodic orbits in $\Lambda_E$, so that each one of them admits infinitely many homoclinics.

\section{Area preserving homeomorphisms of the open annulus}\label{sec_coincide}

Now we deal with the case where the branches of $W^s(P_{2,E})$ and $W^u(P_{2,E})$ inside $S_E$ coincide. In this situation the circles $C^u_N=\Psi^{N-1} \circ \Psi^g(C^u_{\tau_0})$ and $C^s_{\tau_1}$ coincide in $D_{\tau_1}$, where $N\in\N^*$ is as in Proposition \ref{prop_homoclinic2}. Once again we follow notations established in Sections \ref{sec_existenciahomoc} and \ref{sec_twist}.
It immediately follows that $P_{2,E}$ admits infinitely many homoclinics in $\dot S_E$. Moreover, in this case the Hamiltonian flow defines a symplectomorphism $\Psi: \A \to \A$, where $\Psi:=\Psi^g \circ \Psi^l$ and $\A$ consists of $D_{\tau_1}$ with $N$ disjoint closed disks removed
$$\A:= D_{\tau_1} \setminus \bigcup_{j=0}^{N-1} \Psi^{-j}\left(B^s_{\tau_1}\right).$$
Note that the assumption $C^u_N=C^s_{\tau_1}$ implies that $\Psi^{-j}\left(B^s_{\tau_1}\right)=\Psi^{N-1-j} \circ \Psi^g\left(B^u_{\tau_0}\right)$ for all $j=0, \ldots, N-1$. The map $\Psi$ describes the dynamics on an invariant open subset $\U_\A \subset \dot S_E$, which is given by the trajectories in $\dot S_E$ intersecting $\A$ transversely. Thus the Hamiltonian flow restricted to $\U_\A$ admits $\A$ as a global surface of section where $\Psi$ is the first return map. Observe that $\Psi$ preserves the outer boundary component of $\A$ (corresponding to the binding $P_{3,E}$), i.e., $\Psi$ maps points close to $P_{3,E}$ into points close to $P_{3,E}$. The inner boundary components of $\A$ are permuted by $\Psi$.

%Since the local transition map $\Psi^l$ has an infinite twist near $C^s_{\tau_1}=\partial B^s_{\tau_1}$, as in Lemma \ref{twist}, we see that $\Psi^N$ has an infinite twist near each boundary component $\Psi^{-j}\left(C^s_{\tau_1}\right)$ of $\A$, $j=0, \ldots, N-1$.
Using %this fact and
results due to J. Franks on area preserving homeomorphisms of the open annulus \cite{franks1, franks2,franks_erratum}, we shall prove that $\Psi:\A \to \A$ has infinitely many periodic points, which correspond to infinitely many periodic orbits of the Hamiltonian flow in $\dot S_E$.

We start by defining the following equivalence relation in $D_{\tau_1}$:
$$x \sim y \Leftrightarrow \exists\, j\in \{0,\ldots, N-1\} \mbox{ s.t. } x,y \in\Psi^{-j}\left(B^s_{\tau_1}\right).$$
In this way we obtain an open disk $\widetilde D_1:=D_{\tau_1}/_\sim$ whose topology is induced by the natural projection $\Pi: D_{\tau_1} \to \widetilde D_1$. Observe that $\Pi$ collapses the closed disks $\Psi^{-j}\left(B^s_{\tau_1}\right) \subset D_{\tau_1}$ into distinct points $p_j \in \widetilde D_1$, $j=0, \ldots, N-1$. Notice also that $\Pi|_\A$ is a bijection onto $\widetilde D_1 \setminus \{p_0,\ldots,p_{N-1}\}$ and hence $\Pi(\A)$ admits a natural  smooth structure so that $\Pi|_\A$ is a smooth diffeomorphism onto $\Pi(A)$.  Consider the finite area form $\omega_1$ on $\Pi(\A)=\widetilde D_1\setminus\{p_0,\ldots,p_{N-1}\}$ induced by $(\Pi|_\A)_* \omega_0$, where $\omega_0$ is the symplectic form on $D_{\tau_1}$ induced by the ambient symplectic structure.
Using $\Pi$ and the symplectomorphism $\Psi:\A \to \A$ we obtain an area preserving homeomorphism $\Phi: \widetilde D_1 \to \widetilde D_1$ defined on $\Pi(\A)$ by $\Phi = \Pi \circ \Psi \circ \Pi^{-1}$ and by declaring  that $\Phi(p_j)=p_{j-1}, \forall j=1,\ldots,N-1$ and $\Phi(p_0)=p_{N-1}$. The area form $\omega_1$ naturally extends to a finite area form on $\widetilde D_1$ which is $\Phi$-invariant.
%and  has an infinite twist near each of its fixed points $p_j$, $j=0, \ldots, N-1$.

Let us first assume that $N>1$. Note that $\Phi$ must have a fixed point $\bar p\not\in\{p_0,\ldots,p_{N-1}\}$ and $\Phi|_{\widetilde D_1 \setminus\{\bar p\}}$ is an area preserving homeomorphism of the open annulus (homotopic to the identity) with a periodic orbit $\{p_0,\ldots,p_{N-1}\}$. In this case we can directly apply  Franks' theorem in \cite{franks2} to conclude that $\Phi$ has infinitely many periodic points in $\widetilde D_1$ which implies that  the Hamiltonian flow admits infinitely many periodic orbits in $\U_\A \subset \dot S_E$.
%Observe that in the case $N>1$ we have not used the infinite twist near each $p_j$ in order to obtain infinitely many periodic points for $\Phi$.%Pick any $j\in \{0, \ldots, N-1\}$ and consider the restriction of $\Phi$ to the open annulus $\tilde D_1 \setminus \{p_{j}\}$. Such restriction is an area preserving homeomorphism admitting at least $N-1>0$ fixed points.

Now let us assume $N=1$. In this case $\A$ is given by the open annulus $$D_{\tau_1} \setminus B^s_{\tau_1}\simeq (\R/ 2\pi \Z) \times (0,1)$$ and the symplectomorphism $\Psi=\Psi^g \circ \Psi^l: \A \to \A$ can be seen as a diffeomorphism of $(\R / 2\pi \Z) \times (0,1)$ preserving a finite area form.
Recall the local model $(A_c^*,S_c)$ for $(N^s_{\tau_1}, C^s_{\tau_1})$, where $S_c$ and $A_c^*$ are defined in \eqref{Sc} and \eqref{Ac*}, respectively, and $N^s_{\tau_1} \subset D_{\tau_1} \setminus B^s_{\tau_1}$ is defined in \eqref{N}.
In angle-action coordinates $(\theta,I_2)$, $S_c$ is represented by $\{I_2=I_2^c\}$, where $I_2^c>0$ is such that $f(I_2^c)=0$, see \eqref{I1}.
Observe that $(\theta,I_2-I_2^c)\in (\R/ 2\pi \Z) \times (0,\epsilon),$ with $\epsilon>0$ sufficiently small, can be used as coordinates in $\A$ near its inner boundary component $C^s_{\tau_1}$ via the identifications above.

Let $\pi: \R \times (0,1) \to  (\R/2\pi \Z)\times (0,1)$ be the universal covering map
\begin{equation}\label{eq_covmap}
\pi(x,y)=(x \mbox{ mod } 2\pi, y),\forall (x,y)\in \R \times (0,1).
\end{equation}
With the identification above, we have
$$(x,y)=(\theta,I_2-I_2^c), \forall (x,y)\in \R \times (0,\epsilon),$$
where $(\theta,I_2)$ are the angle-action coordinates and $\theta$ is now viewed as $\R$-valued.

Let $\widetilde\Psi: \R \times (0,1) \to \R \times (0,1)$ be a lift of $\Psi$ with respect to $\pi$. The infinite twist behavior of the local transition map proved in Lemma \ref{twist} implies that $\widetilde\Psi$ also has an infinite twist near the lower boundary component $\R\times \{0\}$, as we show in the following lemma.

\begin{lem}\label{lem_global}There exist $0<\epsilon'<\epsilon$ and real-analytic functions $\Theta_1(x),$ defined in $\R$, $H_1(x,y)$ and $H_2(x,y)>0$, defined in $\R \times (-\epsilon',\epsilon')$, all of them $2\pi$-periodic in $x$, such that, denoting $(\tilde x(x,y),\tilde y(x,y))=\widetilde \Psi(x,y)$, we have
\begin{equation}\label{xytil}
\begin{aligned}
\tilde x(x,y) & =x+\Delta(y) +\Theta_1(x+ \Delta(y))+yH_1\left(x+\Delta(y),y\right),\\
\tilde y(x,y) &=y H_2\left(x+\Delta(y),y\right),
\end{aligned}
\end{equation}
for all $(x,y)\in \R\times (0,\epsilon'),$ where
\begin{equation}\label{delta}
\Delta(y):=\frac{\bar \omega(y+I_2^c)}{\bar \alpha(y+I_2^c)} \ln y + \Lambda(y+I_2^c),
\end{equation} and $\bar \alpha, \bar \omega$ and $\Lambda$ are real-analytic functions defined in \eqref{alphabar}, \eqref{omegabar} and \eqref{Lambda}, respectively.
Moreover, $\widetilde \Psi$ satisfies the following infinite twist property near  $\R \times \{0\}:$ fixing a bounded subset $B\subset \R$, we have
$$\tilde x(x,y) \to -\infty \mbox{ and } \tilde y(x,y) \to 0^+, \mbox{ as } y \to 0^+,$$ uniformly for $x\in B$.\end{lem}

\begin{proof} Recall that the local transition map $l$ is given by
$$l(\theta,I_2)=\left(\theta+\frac{\bar \omega(I_2)}{\bar \alpha(I_2)} \ln (I_2-I_2^c)+\Lambda(I_2), I_2\right),$$
where $\bar \alpha, \bar \omega$ and $\Lambda$ are real-analytic functions given by \eqref{alphabar}, \eqref{omegabar} and \eqref{Lambda}, respectively.

In the strip $\R \times (0,\epsilon)$, a point $(x,y)$ coincides with the angle-action coordinates $(\theta,I_2-I_2^c)$ and $\widetilde \Psi$ splits as $\widetilde \Psi = \tilde g \circ \tilde l,$ where $\tilde g$ and $\tilde l$ are lifts of the global and local transition maps, respectively,  with respect to the universal covering map $\pi$, see \eqref{eq_covmap}. In coordinates $(x,y),$ a choice of $\tilde l$ is written as
\begin{equation}\label{ltil}
\tilde l(x,y)=(x+\Delta(y),y), \forall (x,y) \in \R \times (0,\epsilon),
\end{equation}
with $\Delta(y)$ defined in \eqref{delta}.

Let us denote \begin{equation}\label{gtil}\tilde g(x,y)=(X(x,y),Y(x,y)),\end{equation} where $X$ and $Y$ are real-analytic functions defined in $\R \times (-\epsilon,\epsilon)$. Since we are assuming $N=1$, and therefore $\Psi^g(C^u_{\tau_0})=C^s_{\tau_1}$, the global transition map $g$ leaves the boundary component $S_c=\{I_2=I_2^c\}$ invariant. Then $Y(x,0)=0$ for all $x \in \R$ and there exists a real-analytic function $H_2: \R\times (-\epsilon,\epsilon) \to \R$, $2\pi$-periodic in $x$, so that $Y(x,y)=y^k H_2(x,y)$ for some integer $k>0$ and $H_2(x,0)\neq 0$ for all $x\in\R$.
We claim that $k=1$. In fact, since $\tilde g$ is an area preserving diffeomorphism, we have
\begin{equation}\label{det}
\det\left(
\begin{array}{cc}
\partial_x X & \partial_y X\\
\partial_x Y & \partial_y Y
\end{array}
\right)=1.
\end{equation}
Observe that $\partial_x Y(x,0)=0$ for all $x$. Since $\partial_y Y = ky^{k-1} H_2(x,y)+y^k\partial_y H_2(x,y)$ and $\partial_y Y(x,0)\neq 0,\forall x\in \R,$ we must have $k=1$.  Thus
\begin{equation}\label{Y}
Y(x,y)=y H_2(x,y), \forall (x,y)\in \R\times (-\epsilon,\epsilon).
\end{equation}
Using that $\partial_y Y(x,0)>0,$ we obtain $H_2(x,0)>0$ for all $x\in \R$.  This implies the existence of $0<\epsilon'<\epsilon$ so that $H_2(x,y)>0$ for all $(x,y)\in \R\times (-\epsilon',\epsilon')$.

We can take $\epsilon'>0$ even smaller so that $X(x,y)=X(x,0)+y H_1(x,y)$, where $H_1(x,y)$ is a real-analytic function defined in  $\R\times (-\epsilon',\epsilon')$, which is $2\pi$-periodic in $x$. Let $\Theta_1$ be the real-analytic function given by $\Theta_1(x)=X(x,0)-x,$ $\forall x\in\R$. Note that $\Theta_1$ is $2\pi$-periodic. Then we write
\begin{equation}\label{X}
X(x,y)=x+\Theta_1(x)+ yH_1(x,y), \forall (x,y)\in \R\times (-\epsilon',\epsilon').
\end{equation}

Now denoting $(\tilde x,\tilde y)=\tilde g \circ \tilde l$, we finally obtain \eqref{xytil} from \eqref{ltil}, \eqref{gtil}, \eqref{Y} and \eqref{X}. Since $\Delta(y) \to -\infty$ as $y \to 0^+$ and $\Theta_1(x), H_1(x,y)$ and $H_2(x,y)>0$ are bounded, we conclude from \eqref{xytil} that $\tilde x(x,y) \to -\infty$ and $\tilde y(x,y) \to 0^+$ as $y \to 0^+$ for all $x\in\R$. Clearly, this convergence is uniform in $x$, if $x$ lies in a bounded subset $B\subset \R$, since $\Theta_1$ and $H_1$ are bounded.  \end{proof}

For each $k\in\N$, consider the lift $\widetilde\Psi_k: \R \times (0,1) \to \R \times (0,1)$ of $\Psi$ given by
\begin{equation}\label{lift}
\widetilde\Psi_k(x,y)=\widetilde\Psi(x,y)+(2k\pi,0).
\end{equation}
We shall see that the infinite twist of $\widetilde\Psi$ near $\R\times\{0\}$ implies the existence of a fixed point for $\widetilde\Psi_k$ for each $k$ sufficiently large. In order to do that, we first recall the notion of returning disks introduced by J. Franks in \cite{franks1}.
%As a consequence, we obtain infinitely many fixed points for $\Psi$ which correspond to infinitely many periodic orbits of the Hamiltonian flow in $\dot S_E$.

\begin{definition}[J. Franks \cite{franks1}]
Let $f$ be a homeomorphism of the open annulus $A:=(\R / 2\pi \Z) \times (0,1)$, which is homotopic to the identity map. Let $\tilde f$ be a lift of $f$ to $\tilde A:=\R \times (0,1)$ with respect to the covering map $\pi$ defined in \eqref{eq_covmap}. We say that an open disk $U \subset \tilde A$ is a positively returning disk for $\tilde f$ if
\begin{itemize}
\item[(i)] $\tilde f(U) \cap U = \emptyset$,
\item[(ii)] there exists an integer $n>0$ so that $\tilde f^n(U) \cap (U+m) \neq \emptyset$ for some $m>0$, where $U+m:=\{(x+m,y): (x,y) \in U\}$.
\end{itemize}
Similarly, we define negatively returning disks for $\tilde f$ by requiring $m<0$ in \emph{(ii)}.
\end{definition}

The following result proved by Franks in \cite{franks1} (see also \cite{franks_erratum}) provides at least one fixed point for a lift of an area preserving homeomorphism of the open annulus containing both negatively and positively returning disks.

\begin{theo}[J. Franks \cite{franks1,franks_erratum}]\label{teo_franks}
Let $f: A \to A$ be a homeomorphism of the open annulus $A=(\R/ 2\pi \Z) \times (0,1)$ which is homotopic to the identity map and preserves a finite area form. Let $\tilde f: \tilde A \to \tilde A$ be a lift of $f$ to the universal covering space $\tilde A=\R \times (0,1)$ with respect to the covering map $\pi$ defined in \eqref{eq_covmap}.
Assume that $\tilde f$ admits both positively and negatively returning disks, which are lifts of disks in $A$.
%Assume that $\tilde f$ admits both  positively and  negatively returning disks.
Then $\tilde f$ has a fixed point.
\end{theo}

The proof of Theorem \ref{teo_franks} is contained in the proof of the more general Theorem 2.1 in \cite{franks1}, see also \cite{franks_erratum}.

 Proposition \ref{negpos} below asserts that $\Psi$ admits infinitely many fixed points. Its proof consists of showing that $\widetilde \Psi_k$ admits both positively and negatively returning disks, which are lifts of disks, for every $k>0$ sufficiently large. As a consequence of Theorem \ref{teo_franks}, we obtain a fixed point for $\widetilde \Psi_k$ for each $k$ sufficiently large. Observe that fixed points of $\widetilde \Psi_{k_1}$ and $\widetilde \Psi_{k_2}$, with $k_1\neq k_2$,   project into distinct fixed points of $\Psi$. Therefore, we derive infinitely many fixed points for $\Psi$.

\begin{prop}\label{negpos}
$\Psi$ has infinitely many fixed points.
\end{prop}

\begin{proof}
We can assume that $\widetilde \Psi(x,y)=(\tilde x(x,y),\tilde y(x,y))$, $\forall (x,y)\in \R \times (0,\epsilon')$, where $\tilde x$, $\tilde y$ and $\epsilon'>0$ are given in Lemma \ref{lem_global}, see equation \eqref{xytil}.

Define the open disk $Q^+_\mu \subset \R \times (0,\epsilon')$ depending on a small parameter $0<\mu< \frac{\epsilon'}{2}$  as follows
$$Q^+_\mu := (0,2\pi) \times \left(\mu,  \frac{\epsilon'}{2}\right).$$
Using that $\tilde x(x,y) \to -\infty$ and $\tilde y(x,y)\to 0^+$ as $y\to 0^+$, for all fixed $x$, we find $l_0\in \Z,$ so that $\widetilde \Psi(Q^+_\mu) \cap (Q^+_\mu + 2l_0\pi) \neq \emptyset$ , for all $\mu>0$ sufficiently small. Fixing such a small $\mu>0$ and using that $\widetilde\Psi(Q^+_\mu)$ is a bounded subset of $\R \times (0,1)$, we obtain that $\widetilde\Psi_k(Q^+_\mu)$ does not intersect $Q^+_\mu$ for all $k\in\N$ sufficiently large. Now since $\widetilde\Psi_k(Q^+_\mu)$ intersects $Q^+_\mu +2(l_0+k)\pi$ we conclude that $Q^+_\mu$ is a positively returning disk  for $\widetilde\Psi_k$ for all $k\in \N$ large. See Figure \ref{fig_pos_returning}.

\begin{figure}[ht!]
  \centering
  \includegraphics[width=1\textwidth]{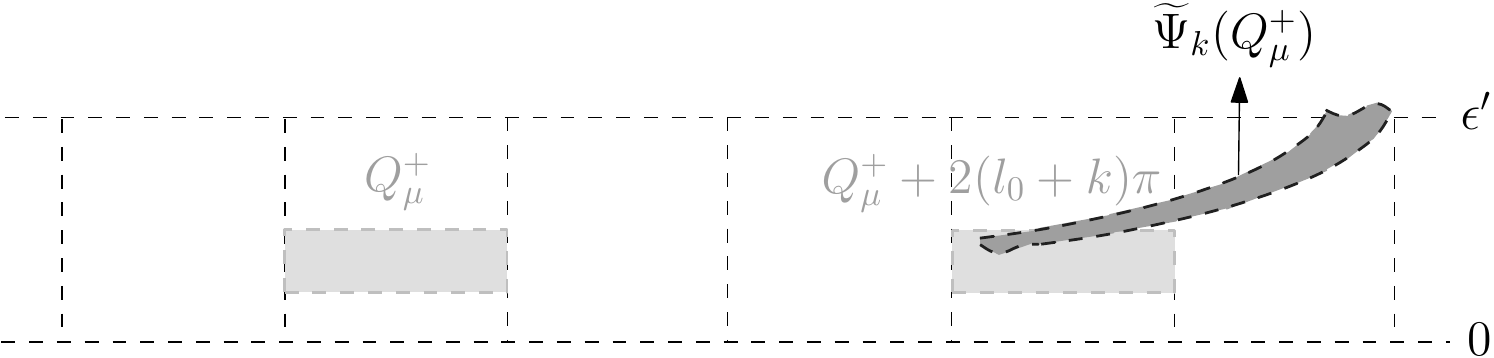}
  \caption{Positively returning disk $Q^+_\mu$ for $\widetilde\Psi_k$.}
  \label{fig_pos_returning}
\hfill
\end{figure}

Now for each fixed $k\in \N$, let $Q^-_{k}\subset \R \times (0, \epsilon')$ be the open disk given by
$$Q^-_{k} := (0,2\pi)\times (0,\varepsilon_k),$$
where $0<\varepsilon_k\ll \mu$ is to be determined below and $\mu$ is fixed as above. Using that $\tilde x(x,y) \to -\infty$ and $\tilde y(x,y)\to 0^+$ as $y\to 0^+$, uniformly in $x\in (0,2\pi)$, we obtain $\varepsilon_k>0$ small so that $\widetilde \Psi(Q^-_k) \cap (Q^-_k-2k\pi)= \emptyset$. Moreover, $\widetilde \Psi(Q^-_k) \cap (Q^-_k - 2l \pi)\neq \emptyset$ for all $l\in \N$ sufficiently large. Fixing one such  $\varepsilon_k>0$ small, we conclude that $\widetilde \Psi_k(Q^-_k) \cap Q^-_k = \emptyset$ and that $\widetilde \Psi_k(Q^-_k) \cap (Q^-_k - 2(l-k)\pi) \neq \emptyset$ for every large $l$. It follows that $Q^-_k$ is a negatively returning disk for $\widetilde \Psi_k$, see Figure \ref{fig_neg_returning}.

\begin{figure}[ht!]
  \centering
  \includegraphics[width=1\textwidth]{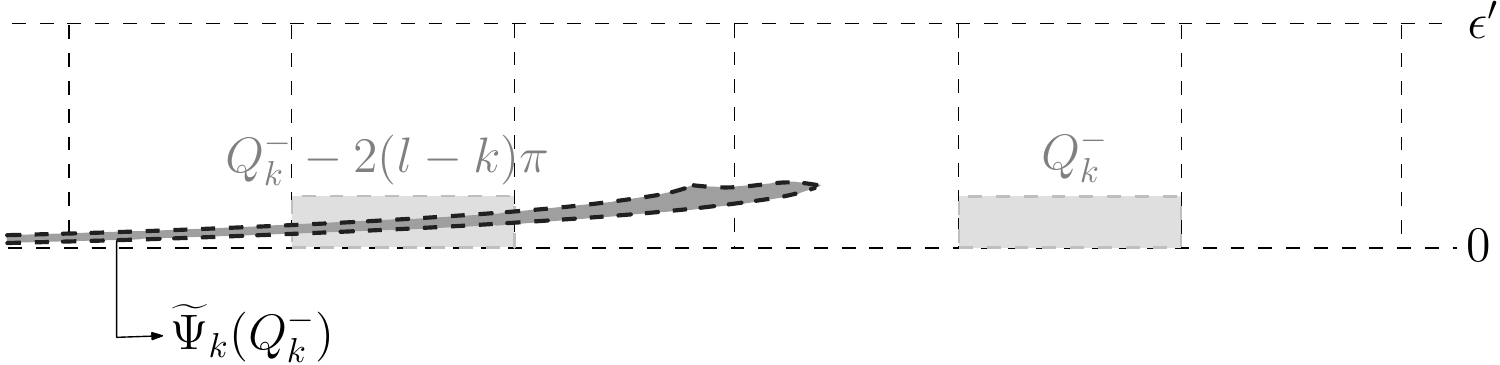}
	  \caption{Negatively returning disk $Q_k^-$ for $\widetilde\Psi_k$.}
  \label{fig_neg_returning}
\hfill
\end{figure}

We conclude that for all $k\in \N$ sufficiently large, $\widetilde\Psi_k$  admits both positively and  negatively returning disks in $\R\times (0,1)$, which are both lifts of disks in $(\R/2\pi \Z)\times (0,1)$. Theorem \ref{teo_franks} implies that $\widetilde\Psi_k$ has a fixed point $q_k \in \R \times (0,1)$. Clearly $\pi(q_k)\in (\R/2\pi \Z)\times (0,1)$ is a fixed point of $\Psi$. Using \eqref{lift} we see that $q_k$ is not a fixed point of $\widetilde\Psi_j$ if $j\neq k$ and hence $\pi(q_j)\neq \pi(q_k)$ if $j\neq k$. This implies that $\Psi$ has infinitely many fixed points. \end{proof}

Finally observe that each fixed point of the first return map $\Psi:\A \subset D_{\tau_1} \to \A$ obtained in Proposition \ref{negpos} corresponds to a periodic orbit of the Hamiltonian flow inside $\U_\A\subset \dot S_E$. Hence $\dot S_E$ contains infinitely many periodic orbits and infinitely many homoclinic orbits to $P_{2,E}$ in case the branches of $W^s_E(P_{2,E})$ and $W^u_E(P_{2,E})$ in $S_E$ coincide.

%Each such periodic orbit forms a Hopf link with the binding orbit $P_{3,E}$ of the $2-3$ foliation.

%We can finally conclude that $\dot S_E$ contains infinitely many homoclinic orbits to $P_{2,E}$ and infinitely many periodic orbits in case the branches of $W^s_E(P_{2,E})$ and $W^u_E(P_{2,E})$ inside $S_E$ coincide.

\vspace{.6cm}

%We shall see that the infinite twist of $\widetilde\Psi$ near $\R\times\{0\}$ implies the existence of a fixed point for $\widetilde\Psi_k$ for each $k$ sufficiently large. First we recall the notion of returning disks introduced by J. Franks in \cite{franks1}.
%As a consequence, we obtain infinitely many fixed points for $\Psi$ which correspond to infinitely many periodic orbits of the Hamiltonian flow in $\dot S_E$.

\noindent {\it Acknowledgements.} NP was partially supported by FAPESP 2014/08113-1. PS was partially supported by CNPq 306106/2016-7 and by FAPESP 2016/25053-8.

\end{document}